\documentclass[graybox]{svmult}

\usepackage{mathptmx}       
\usepackage{helvet}         
\usepackage{courier}        
\usepackage{type1cm}        
\usepackage{makeidx}         
\usepackage{graphicx}        
\usepackage{multicol}        
\usepackage[bottom]{footmisc}
\usepackage{amsmath,amssymb,bm} 

\makeindex             

\usepackage{amsmath,amssymb,bm}

\newcommand{\1}{{\mathbf{1}}} 
\newcommand{\cA}{{\mathcal{A}}} 
\newcommand{\cB}{{\mathcal{B}}} 
\newcommand{\cC}{{\mathcal{C}}}

\newcommand{\cF}{{\mathcal{F}}}

\newcommand{\cK}{{\mathcal{K}}} 
 
\newcommand{\cM}{{\mathcal{M}}} 
\newcommand{\cN}{{\mathcal{N}}} 
 
\newcommand{\cP}{{\mathcal{P}}}

\newcommand{\cS}{{\mathcal{S}}} 
\newcommand{\cT}{{\mathcal{T}}}

\newcommand{\cX}{{\mathcal{X}}}

\newcommand{\N}{{\mathbb{N}}} 
\newcommand{\R}{{\mathbb{R}}} 

\renewcommand{\E}{{\mathbf{E}}} 
\renewcommand{\P}{{\mathbf{P}}} 

\DeclareMathOperator{\Val}{{\rm Val}}

\newcommand{\e}{{\rm{e}}}
\newcommand{\fed}{\,\rule{.1mm}{.24cm}\rule{.24cm}{.1mm}\,}
\newcommand{\bQ}{{\bf{Q}}}
\newcommand{\m}{{\bf {m}}}
\newcommand{\s}{{\mathbb S}}

\DeclareMathOperator{\mix}{mix}
\DeclareMathOperator{\type}{type}
\begin{document}

\title*{Valuations and Boolean Models}
\author{Julia H\"orrmann and Wolfgang Weil}
\institute{Julia H\"orrmann \at Research Group Stochastics/Didactics, Ruhr University Bochum, 44780 Bochum, Germany, \email{julia.hoerrmann@rub.de}
\and Wolfgang Weil \at Institute of Stochastics, Karlsruhe Institute of Technology, 76128  Karlsruhe, Germany, \email{wolfgang.weil@kit.edu}}
\maketitle

\abstract*{To be included later.}

\abstract{Valuations, as additive functionals, allow various applications in Stochastic Geometry, yielding mean value formulas for specific random closed sets and processes of convex or polyconvex particles. In particular, valuations are especially adapted to Boolean models, the latter being the union sets of Poisson particle processes. In this chapter, we collect mean value formulas for scalar- and tensor-valued valuations applied to Boolean models under quite general invariance assumptions.}

\section{Introduction}
\label{sec:1}

Hadwiger's characterization theorem for the intrinsic volumes (see \cite[Theorem 14]{S2015}) has important applications in integral geometry. Besides a kinematic formula for arbitrary continuous valuations on $\cK^n$, the celebrated principal kinematic formula was proved by Hadwiger using his characterization result. In its general form, the principal kinematic formula for the intrinsic volumes $V_j$ reads
\begin{align}\label{PKF}
\int_{G_n} V_j(K\cap gM)\mu(dg) = \sum_{k=j}^n c_{n,j,k} V_k(K)V_{n+j-k}(M),
\end{align}
for convex bodies $K,M\in\cK^n$, $j=0,\dots ,n$, and with given constants $c_{n,j,k}\ge 0$. In 1959, Federer proved a local version of \eqref{PKF}, for curvature measures, a notion he invented on the larger class of sets with positive reach. For both results, the global formula \eqref{PKF} and its local analog for curvature measures, a more direct proof was given in \cite{SW86} by splitting the integration over the motion group $G_n$ into a translation integral and a subsequent integration over the rotation group. This approach yielded also translative integral formulas for intrinsic volumes and curvature measures, introducing certain expressions of mixed type.

The need for translative integral formulas arose  with the development of stochastic geometry in the 1970s by Matheron and Miles (we refer here and in the following to the book \cite{SW}, for more details and specific references). In particular, in two important papers by Miles and Davy in 1976, the problem was discussed how geometric mean values for particles in a randomly overlapping system can be estimated from measurements at the union set. The formulas, which they proved, modelled the particle system by a stationary and isotropic Poisson process of convex bodies (a random countable subset of $\cK^n$ with rigid motion invariant distribution) and then used the principal kinematic formula. As a surprising result, in two and three dimensions, it was possible to estimate the mean number of particles (per unit volume) in an overlapping particle system by measuring the specific area, boundary length and Euler characteristic of the union set in a bounded planar sampling window (respectively, the volume, surface area, integral mean curvature and Euler characteristic in the spatial situation). In addition, also mean particle quantities were obtained (mean area and boundary length, respectively, mean volume, mean surface area and mean integral mean curvature). Such overlapping particle systems occurred frequently in microscopic investigations and became more and more important for techniques in image analysis. There, the random set model, given as the union of a Poisson particle process, the {\it Boolean model}, was not only used for systems with given real particles but also for spatially homogeneous random structures without that there were particles in the background. Then, the mean particle characteristics served as important parameters to find an appropriate distribution for a fictive particle process to adjust a Boolean model to the given structure.

For such applications, the assumption of stationarity (spatial homogeneity) was mostly acceptable, but the isotropy (rotation invariance) was often not fulfilled. This initiated the study of non-isotropic Boolean models, for which translative integral formulas were needed. In general dimensions, and for the intrinsic volumes, a corresponding system of formulas for stationary Boolean models was presented in \cite{W90}. A further important step was made in \cite{Fallert} by showing that the translative integral formulas for intrinsic volumes, in their local form for curvature measures, even produced mean value results in the non-stationary case. In the subsequent years, many related integral-geometric results were obtained, for mixed volumes, support functions, area measures, and applied to particle processes and Boolean models. Recently, translative integral formulas for general valuations $\varphi\in\Val$ and local versions for measure-valued valuations became available and corresponding mean value formulas for Boolean models were established in \cite{W2015a, W2015b}. As we shall see, some of these results can also be applied to tensor valuations.

In the following survey, we describe the interrelations between valuations, translative integral formulas and Boolean models and give appropriate references. We mostly concentrate on the stationary situation. After discussing the general results, we collect various examples in Section 6. The final Section 7 describes shortly extensions to non-stationary structures and also gives an outlook on the use of harmonic intrinsic volumes in the analysis of non-isotropic Boolean models, a development which was started in \cite{Hoerr}. In two and three dimensions measurements of the specific harmonic intrinsic volumes allow the estimation of the mean number of particles per unit volume and of the mean harmonic intrinsic volumes (which include in particular the usual intrinsic volumes). Thus, these recent results are a natural extension of the formulas by Miles and Davy from 1976 to the non-isotropic situation.

\section{Basic Definitions and Background Information}
\label{sec:2}

We consider the space $\cK^n$ of convex bodies in $\R^n, n\ge 2$, supplied with the Hausdorff metric $\delta(\cdot,\cdot)$ and the dense subset $\cP^n$ of convex polytopes. In the following, we study {\it valuations} on $\cK^n$ or $\cP^n$. These are mappings $\varphi :\cK^n \to \cX$ (or $\varphi : \cP^n\to \cX$), where $\cX$ is a commutative (topological) semigroup and which are additive in the sense that
$$
\varphi (K\cup M) + \varphi (K\cap M) = \varphi (K) + \varphi (M) ,
$$
whenever $K,M$ and $K\cup M$ lie in $\cK^n$ (respectively, in $\cP^n$). We concentrate on the situations where $\cX =\R$ (real valuations), $\cX = \cM(\R^n\times \s^{n-1})$ (measure valuations; here $\cM(\R^n\times \s^{n-1})$ is the space of finite signed Borel measures on $\R^n\times \s^{n-1}$),  and $\cX = \cT^n$ the space of tensors in $\R^n$ (tensor valuations).

\subsection{Real valuations}

Concerning real valuations, we concentrate on the class $\Val$ of translation invariant continuous valuations, in the following. The standard examples of valuations in $\Val$ are the {\it intrinsic volumes} $V_m(K), m=0,\dots ,n$, for $K\in\cK^n$. They are, in addition, invariant under rotations. McMullen \cite{McM77,McM93} has shown that every valuation $\varphi\in\Val$ admits a (unique) decomposition
\begin{equation}\label{val}
\varphi = \sum_{j=0}^{n} \varphi_j
\end{equation}
into $j$-homogeneous valuations $\varphi_j$ (which are again translation invariant and continuous). Here, $\varphi_0$ is a constant and
Hadwiger \cite{Had52} has proved that $\varphi_n = c_nV_n$. For $m=1,...,n-1$,  the vector space ${\Val}_m$ of $m$-homogeneous valuations is infinite-dimensional. In particular, McMullen \cite{McM80} has shown that $\varphi\in {\Val}_{n-1}$, if and only if
\begin{equation}\label{(n-1)case}
\varphi (K) = \int_{\s^{n-1}} f(u) S_{n-1} (K,du),\quad K\in{\cK}^n,
\end{equation}
for some continuous function $f= f_\varphi$ on $\s^{n-1}$ which is uniquely determined, up to a linear function (see \cite[Theorem 15]{S2015}). Here, $S_{n-1} (K,\cdot)$ is the area measure of $K$.

We also recall from \cite[Theorem 10]{S2015} that, for a polytope $P\in{\cal P}^n$ and $\varphi_j\in\Val_j, j=1,\dots ,n-1,$ we have
\begin{equation}\label{pol}
\varphi_j(P) = \sum_{F\in{\cal F}_j(P)} f_j(n( P,F)) V_j(F) ,
\end{equation}
where the summation is over all $j$-dimensional faces of $P$, $n(P,F)$ is the set of all unit vectors which are normals of $P$ at relative interior points of $F$ ($n(P,F)$ is a spherical polytope of dimension $n-j-1$) and $f_j$ is a simple valuation on the spherical polytopes of dimension $\le n-j-1$.

\subsection{Measure valuations}

Concerning measure valuations, we mention the {\it support measures} $\Lambda_j(K,\cdot), j=0,\dots ,n-1$, which are finite measures on $\R^n\times \s^{n-1}$, continuous with respect to the weak convergence of measures and  translation covariant in the sense that
$$
\Lambda_j(K,A\times B) = \Lambda_j(K+x,(A+x)\times B)
$$
for Borel sets $A\subset\R^n, B\subset \s^{n-1},$ and all $x\in\R^n$. They are also rotation covariant,
$$
\Lambda_j(K,A\times B) = \Lambda_j(\vartheta K,\vartheta(A\times B)), \quad \vartheta\in SO_n.
$$

For a continuous, translation covariant measure valuation $\varphi : \cK^n\to \cM (\R^n\times \s^{n-1})$ there is a decomposition similar to \eqref{val} proved in \cite{KW}. Namely,
\begin{equation}\label{measval}
\varphi = \sum_{j=0}^{n} \varphi_j
\end{equation}
with $j$-homogeneous measure valuations $\varphi_j$ (which are again translation covariant and continuous). Here, homogeneity means that
\begin{equation}\label{meashom}
\varphi(\alpha K,(\alpha A)\times B) = \alpha^j\varphi (K,A\times B)
\end{equation}
for Borel sets $A\subset\R^n, B\subset \s^{n-1},$ and all $\alpha\ge 0$.

The support measures give rise to two further series of measures, the {\it curvature measures} $C_0(K,\cdot),\dots ,C_{n-1}(K,\cdot)$ and the {\it area measures} $S_0(K,\cdot),\dots ,S_{n-1}(K,\cdot)$ of $K$. The former are (up to some constant) the projections of the support measures onto the first component and the latter are (up to the same constant) the projections onto the second component. For different normalizations of curvature and area measures, see \cite[Section 4]{S2015}.


\subsection{Tensor valuations}

The basic tensor valuations, the {\it Minkowski tensors} $\Phi_j^{r,s}(K)$ (of rank $r+s$), arise as (tensor) integrals of the support measures,
$$
\Phi_j^{r,s}(K) =c_{n-j}^{r,s} \int_{\R^n\times \s^{n-1}} x^ru^s \Lambda_j(K,d(x,u))
$$
where  $c_{k}^{r,s}:=\frac{1}{r!s!}\frac{\omega_{k}}{\omega_{k+s}}$ for $k\in\{1,\ldots,n\}$.
The Minkowski tensors have a covariance property with respect to translations (see \cite{HS15}).

\section{The Basic Equation for Boolean Models}
\label{sec:3}

The {\it Boolean model} is a random closed set $Z\subset\R^n$ which arises in a special way, as the union of sets (called {\it grains}) from a Poisson process $Y$. Usually, the grains are assumed to be compact or even compact and convex. More general random sets $Z$ can be considered, if $Y$ is an arbitrary point process on the class ${\cal C}^n$ of nonempty compact sets in $\R^n$ or on the subclass of convex bodies ${\cal K}^n$. In particular, if $Z$ and $Y$ are {\it stationary}, that is, have a distribution invariant under translations, the random set $Z$ can be interpreted as a {\it germ-grain model},
$$
Z:=\bigcup_{i=0}^\infty (x_i+Z_i) ,$$
 where points ({\it germs}) $x_1,x_2,\dots $ are distributed in $\R^n$ according to a stationary point process $X$ and then random compact (or compact, convex) sets $Z_i$ (the grains) are attached to the germs in a suitable way. We shall describe this construction in the next subsection, but will concentrate on the Poissonian case and convex sets, that is to Boolean models, where the grains are convex and independent from each other and from the underlying germ process $X$. These strong independence properties together with the fact that the realizations of $Z$ are locally polyconvex allow to apply valuations $\varphi$ to $Z$ and to express the expected value ${\E}\varphi (Z\cap K_0)$ in a bounded sampling window $K_0$ by the characteristic parameters of $X$ and the $Z_i$. This will be explained in the second subsection. The effective further investigation of Boolean models then requires formulas from Translative Integral Geometry, as they will be provided in Section 4. Background material on random sets, point processes and the integral geometric results as well as further material on Boolean models can be found in \cite{SW} and we refer to this book for all details which are not explained in the following.

\subsection{Boolean models}

Since we will only consider stationary Boolean models $Z$ with convex grains throughout the following, we start with a stationary Poisson process in $\R^n$. A stationary {\it point process} $X$ in $\R^n$ is a (simple) random counting measure $$X:=\sum_{i=1}^\infty \delta_{\xi_i},$$
 where $\delta_x$ denotes the Dirac measure in $x\in\R^n$ and where the $\xi_i$ are distinct random points in $\R^n$. We also assume that $X$ is locally finite meaning that ($\P$-almost surely) each $C\in \cC^n$ contains only finitely many points $\xi_i$ from $X$. Here, in the description, we already made use of the fact that such a point process can be represented in an alternative way, as a locally finite (random) closed set
 $$X=\{\xi_1,\xi_2,\dots \}\subset\R^n.$$
To make these definitions precise, we need an underlying probability space $(\Omega ,\cA,\P)$ and $\sigma$-algebras on the class $\cF^n$ of closed sets in $\R^n$, respectively on the class $\sf N$ of counting measures in $\R^n$. For details, we refer to \cite{SW} but mention that the former is chosen as the Borel $\sigma$-algebra $\cB(\cF^n)$ of the hit-or-miss topology on $\cF^n$  and the latter, $\cN$, is generated by the evaluation (or counting) maps
$$
\Phi_A : \eta \mapsto \eta (A),\quad A\in\cB (\R^n).
$$
The stationarity, which we assume in addition, means that $X+t$ has the same distribution as $X$, for all translations $t\in\R^n$. $X$ is {\it isotropic}, if the distribution of $X$ is invariant under rotations. (Here translations and rotations act in the natural way on counting measures, respectively on closed sets).

The (stationary) point process $X$ is a Poisson process, if $X(A)$ has a Poisson distribution for all bounded Borel sets $A\subset\R^n$,
$$
\P (X(A)=k) = \e^{-\gamma\,\lambda_n(A)}\frac{(\gamma\,\lambda_n(A))^k}{k!},\quad k=0,1,2,\dots
$$
Here $\gamma = \E X([0,1]^n)$ is the {\it intensity} of the Poisson process. It describes the mean number of points of $X$ per unit volume. Because of the stationarity we have $\E X(A) = \gamma\, \lambda_n(A)$, for all $A\in\cB(\R^n)$. As a consequence of the Poisson property, the random variables $X(A_1),\dots ,X(A_m)$ are (stochastically) independent if the Borel sets $A_1,\dots, A_m$ are disjoint. More generally, in this case, also the restrictions $X\fed A_1,\dots , X\fed A_m$ are independent random measures. The Poisson process $X$ is uniquely determined (in distribution) by the parameter $\gamma$. Since the Lebesgue measure $\lambda_n$ is rotation invariant, $X$ is isotropic.

Now assume that $X$ is a stationary Poisson process with intensity $\gamma >0$, enumerated (in a measurable way) as $X=\{\xi_1,\xi_2,\dots\}$. Let $\bQ$ be a probability measure on $\cK^n$ (supplied with the Borel $\sigma$-algebra with respect to the Hausdorff metric) and let $Z_1,Z_2,\dots$, be a sequence of independent random convex bodies with distribution $\bQ$ (and independent of the Poisson process $X$). Then
$$
Z= \bigcup_{i=1}^\infty (\xi_i+Z_i)
$$
is a stationary random set, a {\it Boolean model}. Some additional assumptions are helpful. First, we require that
\begin{equation}\label{intcond}
\int_{\cK^n} V_n(K+B^n)\,\bQ(dK) <\infty ,
\end{equation}
since then $Z$ is a closed set (and moreover $Z\cap K$ is polyconvex for each $K\in{\cK^n}$). Second, we assume that $\bQ$ is concentrated on the {\it centered} convex bodies $\cK^n_c$ (the class of bodies $K\in\cK^n$ with center of the circumsphere at the origin). The effect of this condition is that $\bQ$ is uniquely determined by $Z$ and that $Z$ is isotropic if and only if $\bQ$ is invariant under rotations.

For the following it is often convenient to use the particle process $Y=\{\xi_1+Z_1,\xi_2+Z_2,\dots\}$. This is a point process on the locally compact space $\cK^n$ (that is, a (simple) random counting measure on $\cK^n$ or, equivalently, a locally finite random closed subset of $\cK^n$). The process $Y$ also has the  Poisson property, that means, the random number $Y(A)$ of particles from $Y$ in a Borel set $A\subset\cK^n$ has a Poisson distribution. Later we will use this for the sets
$$
\cK_C := \{ K\in\cK^n : K\cap C\not= \emptyset\},\quad C\in\cC^n.
$$

\begin{proposition}\label{3.1.1} For $A\in\cB(\cK^n)$, we have
$$
\P (Y(A)=k) = \e^{-\Theta (A)}\frac {(\Theta (A))^k}{k!},\quad k=0,1,\dots ,
$$
where $\Theta$ is the image measure (on $\cK^n$) of $\gamma\,\lambda_n\otimes\bQ$ under the mapping $\Phi : \R^n\times\cK^n_c\to \cK^n, (x,K)\mapsto x+K$.
\end{proposition}

\begin{proof}\smartqed
By the extension theorem of measure theory, and since $\Phi$ is a homeomorphism, it is sufficient to prove the result for sets $A=\Phi (B\times C), B\in {\cB}(\R^n), C\in {\cB}(\cK^n_c)$. In this case, the independence properties of $Z$ yield
\begin{align*}
\P (Y(A)=k) &= \P \left(\sum_{i=1}^\infty \1\{\xi_i\in B, Z_i\in C\} =k\right)\\
&= \sum_{j=k}^\infty \P (X(B)=j){j\choose k}\bQ (C)^k(1-\bQ (C))^{j-k}\\
&= \e^{-\gamma\,\lambda_n(B)}\bQ (C)^k\sum_{j=k}^\infty {j\choose k}(1-\bQ (C))^{j-k}\frac{(\gamma\,\lambda_n(B))^j}{j!}\\
&= \e^{-\gamma\,\lambda_n(B)}\bQ (C)^k\frac{(\gamma\,\lambda_n(B))^k}{k!}\sum_{i=0}^\infty (1-\bQ (C))^{i}\frac{(\gamma\,\lambda_n(B))^i}{i!}\\
&= \e^{-\gamma\,\lambda_n(B)}\frac{(\gamma\,\lambda_n(B)\bQ (C))^k}{k!} \e^{\gamma\,\lambda_n(B)-\gamma\,\lambda_n(B)\bQ (C)}\\
&= \e^{-\Theta(A)}\frac{(\Theta (A))^k}{k!} .
\end{align*}
\qed\end{proof}

We emphasize the fact that the (stationary) Boolean model $Z$ is uniquely determined (in distribution) by the two quantities $\gamma$ (a constant which we always assume to be $>0$) and $\bQ$ (a probability measure on $\cK^n_c$). Thus, in order to fit a Boolean model to given data (in form of closed sets in a window $K_0$, say), one has to determine (more precisely, to estimate) $\gamma$ and $\bQ$ from the data, that is from observations of realizations $Z(\omega)\cap K_0$ of $Z$ in $K_0$.

\subsection{Additive functionals}

Concerning the estimation problem described above, let us assume that we observe $\varphi (Z(\omega)\cap K_0)$ for some geometric functional $\varphi$ in the window $K_0$. The mean value $\E \varphi (Z\cap K_0)$ is then the quantity which can be estimated unbiasedly by $\varphi (Z(\omega)\cap K_0)$. It is natural to assume that the window is convex, hence $K_0\in\cK^n$. Then $Z\cap K_0$ is polyconvex a.s. Therefore it is another natural assumption that $\varphi$ is additive, hence a valuation. Since $Z$ is stationary, the location of the window should not matter, therefore we may also assume $\varphi$ to be translation invariant. Finally, to have a smooth behavior with respect to approximations (at least on $\cK^n$), we assume that $\varphi$ is continuous on $\cK^n$. Hence, we consider now $\E \varphi (Z\cap K_0)$, for $K\in\cK^n$ and $\varphi\in\Val$.

The following result from \cite{WW} (see also \cite[Theorem 9.1.2]{SW}), expresses $\E \varphi (Z\cap K_0)$ in terms of $\gamma$ and $\bQ$.

\begin{theorem}\label{Th2}
Let $Z$ be a stationary Boolean model in $\R^n$ with convex grains, let $K_0\in\cK^n$ be a sampling window and $\varphi\in\Val$. Then $\E |\varphi (Z\cap K_0)|<\infty$ and
$$
\E \varphi (Z\cap K_0) = \sum_{k=1}^\infty \frac{(-1)^{k-1}}{k!}\gamma^k\int_{\cK^n_c}\cdots \int_{\cK^n_c}
\Phi(K_0,K_1,\dots,K_k)\,\bQ(dK_1)\cdots \bQ(dK_k)
$$
with
\begin{align}
\Phi(K_0,K_1,\dots,K_k):= &\int_{(\R^n)^k}\varphi (K_0\cap (K_1+x_1)\cap\dots\nonumber\\
&\ldots \cap(K_k+x_k)) \, \lambda_n^k(d(x_1,\dots ,x_k)) \, .\label{itint}
\end{align}
\end{theorem}

\begin{proof} \smartqed We sketch the proof since it sheds some light on the role of the Poisson assumption underlying the Boolean model. To simplify the formulas, we use the particle process $Y=\{\xi_1+Z_1,\xi_2+Z_2,\dots\}$.

Almost surely the window $K_0$ is hit by only finitely many grains $M_1,\dots ,M_\nu\in Y$ (here $\nu$ is a random variable). The additivity of $\varphi$ implies
\begin{align}
\varphi (Z\cap K_0) &= \varphi \left(\bigcup_{i=1}^\nu M_i \cap K_0\right)\nonumber\\
&= \sum_{k=1}^\nu (-1)^{k-1} \sum_{1\le i_1<\dots<i_k\le\nu}\varphi (K_0\cap M_{i_1}\cap\dots\cap M_{i_k})\nonumber\\
&= \sum_{k=1}^\infty \frac{(-1)^{k-1}}{k!}\sum_{(N_1,\dots ,N_k)\in Y^k_{\not=}} \varphi (K_0\cap N_{1}\cap\dots\cap N_{k}).\label{inclexcl}
\end{align}
Here, $Y^k_{\not=}$ denotes the set of all $k$-tupels of pairwise distinct bodies in $Y$ and we could extend the summation to infinity since $\varphi (\emptyset) =0$.

The continuity of $\varphi$ on $\cK^n$ implies $|\varphi (M)|\le c(K_0)$ for all $M\in\cK^n, M\subset K_0$. Hence,
$$
|\varphi (Z\cap K_0)| \le \sum_{k=1}^\nu {\nu\choose k}c(K_0) \le 2^\nu c(K_0).
$$
From Proposition \ref{3.1.1} we get
\begin{align*}
\E 2^\nu &= \sum_{k=0}^\infty 2^k \P (Y(\cK_{K_0})=k)\\
&=\e^{-\Theta (\cK_{K_0})} \sum_{k=0}^\infty \frac{(2\Theta (\cK_{K_0}))^k}{k!}\\
&= \e^{\Theta (\cK_{K_0})} <\infty
\end{align*}
since
\begin{align*}
\Theta (\cK_{K_0}) &= \gamma \int_{\cK^n_c}\int_{\R^n} \1\{(x+K)\cap K_0\not=\emptyset\} \,\lambda_n(dx)\bQ(dK)\\
&= \gamma \int_{\cK^n_c}V_n(K+(-K_0))\,\bQ(dK)\\
&\le \gamma \max\{r(K_0),1\}^n \int_{\cK^n_c}V_n(K+B^n)\, \bQ(dK) <\infty ,
\end{align*}
due to condition \eqref{intcond} (here $r(K_0)$ is the circumradius of $K_0$). Hence $\E |\varphi (Z\cap K_0)|<\infty$.

This integrability property allows to use the dominated convergence theorem for $\E \varphi (Z\cap K_0)$, where $\varphi (Z\cap K_0)$ is expressed by formula \eqref{inclexcl}, and interchange expectation and summation. We get
$$
\E \varphi (Z\cap K_0) = \sum_{k=1}^\infty \frac{(-1)^{k-1}}{k!}\E\sum_{(N_1,\dots ,N_k)\in Y^k_{\not=}} \varphi (K_0\cap N_{1}\cap\dots\cap N_{k}).
$$
Now we use the Campell theorem for point processes \cite[Theorem 3.1.2]{SW} (applied to the special point process $ Y^k_{\not=}$ on $(\cK^n_c)^k$) and the fact that the intensity measure of $Y^k_{\not=}$, for a Poisson process $Y$, is the product measure $\Theta^k$ of $\Theta$ (see \cite[Corollary 3.2.4]{SW}). We obtain
\begin{align}
\E \varphi (Z\cap K_0) &= \sum_{k=1}^\infty \frac{(-1)^{k-1}}{k!}\int_{\cK_c^n}\dots \int_{\cK_c^n} \varphi (K_0\cap N_{1}\cap\dots \nonumber\\
&\quad\ldots \cap N_{k})\, \Theta (dN_1)\cdots\Theta (dN_k). \label{basicint}
\end{align}
Inserting the special form of $\Theta$ now yields the result. \qed
\end{proof}

{\bf Remarks.} {\bf 1.} There is a simple and obvious generalization of the last theorem to Boolean models with polyconvex grains, if the integrability condition \eqref{intcond} is modified appropriately (the number of convex bodies which constitute the typical polyconvex grain should be limited). There is also another, less obvious generalization to non-stationary Boolean models. This requires to consider a general Poisson process $Y$ on $\cK^n$ where the intensity measure $\Theta$ can have a more general form (this induces that the underlying Poisson process $X$ in $\R^n$ is also not stationary anymore). Formula \eqref{basicint} then still holds, provided $\Theta$ is {\it translation regular}. We will explain this and give more results in Section 8.

{\bf 2.} Due to the stationarity and the independence properties of the Poisson process $Y$, the grains $M_1, M_2,\dots \in Y$  are almost surely in general relative position. This implies that geometric functionals $\varphi$ on $\cK^n$ or $\cP^n$ can have an additive extension to the polyconvex set $Y\cap K_0$, although they are not valuations. Examples are the local functionals on   $\cP^n$ considered in \cite{W2015a}. For them, Theorem \ref{Th2} still holds for Boolean models with polytopal grains. A $j$-homogeneous local functional $\varphi_j(P), P\in\cP^n$, of interest is the total content of the $j$-dimensional skeleton of $P$,
$$
\varphi_j(P) = \sum_{F\in{\cF}_j(P)} V_j(F) .
$$
For $j=0,\dots ,n-2$, this functional is not additive on $\cK^n$. Since we concentrate on valuations in this chapter, we will not discuss general local functionals further and refer to \cite{W2015a}, for information (but observe the remarks on local extensions of valuations in Section 4.1).

\section{Integral Geometry for Valuations}
\label{sec:4}

In order to simplify the expectation formula in Theorem \ref{Th2}, it would be helpful to have a more explicit expression for the (iterated) translative integral in \eqref{itint}. This is obtained in the following subsection. In the second subsection, we discuss kinematic formulas.

\subsection{The translative formula}

Some results on translative integrals in dimension 2 and 3 are due to Blaschke, Berwald and Varga in 1937. A first general translative integral formula for intrinsic volumes in $\R^n$ and their local counterparts, the curvature measures, was obtained in \cite{SW86} and the iterated version was proved in \cite{W90} (see the Notes for Section 6.4 in \cite{SW}, for further references, variations and extensions).

The $k$-fold iterated translative integral formula for the intrinsic volume $V_j$ involved mixed functionals which were denoted by $V^{(j)}_{m_1,\dots ,m_k}$ (where the parameters satisfy $m_1+\dots +m_k=(k-1)n+j$), a notation which was subsequently used also for various related results on support measures and other local functionals. In the sequel, we use a special notation, which was introduced in \cite{Hoerr} to simplify the resulting formulas. First, we observe that the exponent $(j)$ in such mixed expressions can be determined from $j= m_1+\dots +m_k-(k-1)n$ and is therefore redundant. Then, we introduce a multi-index $\m = (m_1,\dots ,m_k)$ from the class
\[
\mix(j,k):=\{\m =(m_1,\ldots,m_k)\in \{j,\ldots,n\}^k:m_1+\ldots+m_k=(k-1)n+j\},
\]
 for $j\in\{0,\ldots,n\}$ and $k\in \N$, and abbreviate the mixed functional $\varphi_{m_1,\dots ,m_k}$ by $\varphi_{\m}$.
For $\m\in \mix(j,k)$, we also write $\type(\m):=j$ and $|\m|:=k$.

The following theorem was obtained in \cite{W2015b}, based on a corresponding result for polytopes in \cite{W2015a}.

\begin{theorem}\label{trans4} For $\varphi\in\Val$, let  $\varphi_j$ be its $j$-homogeneous part, $j=0,...,n$, with $\varphi_n =c_nV_n$.  Then, for $k\ge 2$, there exist mixed functionals $\varphi_{\m}$, $\m\in\mix(j,k)$, on ${({\cK^n})^{k}}$ such that for convex bodies $K_1,...,K_k \in {\cK^n}$,
\begin{align}
  \int_{({\mathbb R}^n)^{k-1}}&
\varphi_j(K_1\cap (K_2+{x_2})\cap \dots \cap (K_k+{x_k}))\, \lambda_n^{k-1} (d(x_2,\dots ,x_k))\nonumber \\
& = \sum_{\m\in\mix(j,k)}
\varphi_{\m}(K_1,\dots,K_k).\label{ittrans}
\end{align}
For $(m_1,\ldots,m_k)\in\mix(j,k)$ the mapping $(K_1,...,K_k)\mapsto \varphi_{m_1,\ldots,m_k}(K_1,\dots,K_k)$ is symmetric (w.r.t. permutations of the indices $1,...,k$), it is homogeneous of degree $m_i$ in $K_i$ and it is a valuation in $\Val$ in each of its variables $K_i$.
\end{theorem}

\begin{proof}\smartqed Again, we give only a sketch of the proof and refer to \cite{W2015a, W2015b} for details. In particular, we omit the discussion of the necessary measurability properties. Also, we concentrate on the case $k=2$, the general case follows then by iteration. On the other hand, we give a more general proof using measures, since this will give us also a local version of the theorem, as explained in Remark 1 below.

The main idea is to consider polytopes first and then extend the result to arbitrary convex bodies by approximation. Since the restriction of $\varphi_j$ to $\cP^n$ is a continuous, translation invariant valuation, homogeneous of degree $j$, we can use \eqref{pol} to decompose $\varphi_j(P), P\in \cP^n,$ as
$$
\varphi_j (P) = \sum_{F\in\cF_j(P)} f_j(n (P,F)) V_j(F) .
$$
We define a measure $\Phi_j (P,\cdot)$ on $\R^n$ by
$$
\Phi_j (P,\cdot) := \sum_{F\in\cF_j(P)} f_j(n (P,F)) \lambda_F .
$$
Here, $\lambda_F$ denotes the restriction to $F$ of the $j$-dimensional Lebesgue measure in the affine space generated by $F$.
Then, $P\mapsto \Phi_j (P,\cdot)$ is a measure valuation, homogeneous of degree $j$ (in the sense of \eqref{meashom}) and translation covariant.

The following arguments are similar to those used in the proofs of Theorems 5.2.2 and 6.4.1 in \cite{SW}. Since $f_j$ can be written as a difference of two positive functions $f_j=f_j^+-f_j^-$, we may assume $f_j\ge 0$ (the additivity of $f_j$, which can get lost in this decomposition, does not play a role in the following arguments).

Let $P,Q\in\cP^n$, $A,B\in\cB(\R^n)$ and $x\in\R^n$. Then,
$$
\Phi_j (P\cap (Q+x),A\cap (B+x)) = \sum_{F'\in\cF_j(P\cap (Q+x))} f_j(n (P\cap (Q+x),F')) \lambda_{F'}(A\cap (B+x)) .
$$
For $\lambda_n$-almost all $x$, the face $F'$ is the intersection $F'=F\cap (G+x)$ of some $m$-face $F$ of $P$ with the translate of a $(n+j-m)$-face $G$ of $Q$, $m\in\{j,\dots, n\}$ (such that $F$ and $G+x$ meet in relative interior points). The normal cone of $F\cap (G+x)$ does not depend on the choice of $x$, let $n(P,Q;F,G)$ be its intersection with ${\mathbb S}^{n-1}$. Thus,
\begin{align*}
\int_{\R^n}&\Phi_j (P\cap (Q+x),A\cap (B+x))\, \lambda_n(dx)\\
&= \sum_{m=j}^n\sum_{F\in\cF_m(P)}\sum_{G\in\cF_{n+j-m}(Q)}f_j(n(P,Q;F,G))\int_{\R^n}\lambda_{F\cap(G+x)}(A\cap (B+x))\,\lambda_n(dx).
\end{align*}
In \cite[pp. 185-186]{SW} it is shown that
$$
\int_{\R^n}\lambda_{F\cap(G+x)}(A\cap (B+x))\,\lambda_n(dx) = [F,G]\lambda_F(A)\lambda_G(B) ,
$$
where $[F,G]$ denotes the determinant between $F$ and $G$ (see \cite[p. 183]{SW}). Hence, if we define a measure $\Phi_{m,n+j-m}(P,Q;\cdot)$ on $\R^n\times\R^n$ by
\begin{equation}\label{mixtpol}
\Phi_{m,n+j-m}(P,Q;\cdot) := \sum_{F\in\cF_m(P)}\sum_{G\in\cF_{n+j-m}(Q)}f_j(n(P,Q;F,G))[F,G]\lambda_F\otimes\lambda_G ,
\end{equation}
we arrive at
\begin{align}\label{transintpol}
\int_{\R^n}&\Phi_j (P\cap (Q+x),A\cap (B+x))\, \lambda_n(dx)= \sum_{m=j}^n \Phi_{m,n+j-m}(P,Q;A\times B).
\end{align}

Now we consider the total measures $\Phi_j (P\cap (Q+y),\R^n)$ (which equals our valuation $\varphi_j(P\cap(Q+y))$) and $\varphi_{m,n+j-m}(P,Q):=\Phi_{m,n+j-m}(P,Q;\R^n\times\R^n)$. Then \eqref{transintpol} implies
\begin{align}\label{transintpol3}
\int_{\R^n}\varphi_j (P\cap (Q+x))\, \lambda_n(dx)
&= \sum_{m=j}^n
\varphi_{m,n+j-m}(P,Q).
\end{align}
We remark that
\begin{align}\label{pol1}
\varphi_{m,n+j-m}(P,Q) &= \sum_{F\in\cF_m(P)}\sum_{G\in\cF_{n+j-m}(Q)}f_j(n(P,Q;F,G))[F,G]V_j(F)V_{n+j-m}(G)
\end{align}
and therefore
\begin{equation}\label{pol2}
\varphi_{m,n+j-m}(rP,sQ) = r^ms^{n+j-m}\varphi_{m,n+j-m}(P,Q)
\end{equation}
for $r,s>0$.

We define a functional $J$ on $\cK^n\times\cK^n$ by
$$
J(K,M) := \int_{\R^n}\varphi_j (K\cap (M+x))\, \lambda_n(dx).
$$
Let $K_i\to K,M_i\to M$ be convergent sequences.
Since $\varphi_j$ is continuous, there is a constant $c(K+B^n)$ such that
$$
|\varphi_j(K_i\cap (M_i+x))|\le c(K+B^n) {\bf 1}_{(K+B^n)-(M+B^n)}(x)
$$
for all large enough $i$.
For $\lambda_n$-almost all $x$ the integrand $\varphi_j(K_i\cap (M_i+x))$ converges to $\varphi_j(K\cap (M+x))$ (namely, for all $x$ for which $K$ and $M+x$ do not touch each other). Hence the dominated convergence theorem implies $J(K_i,M_i)\to J(K,M)$, thus $J$ is continuous.

Choosing polytopes $P_i\to K,Q_i\to M$, we obtain $J(rP_i,sQ_i)\to J(rK,sM)$ for all $r,s>0$. Since
$$
J(rP_i,sQ_i) = \sum_{m=j}^n r^ms^{n+j-m}\varphi_{m,n+j-m}(P_i,Q_i),
$$
the coefficients $\varphi_{m,n+j-m}(P_i,Q_i)$ of this polynomial have to converge, and we denote the limits by $\varphi_{m,n+j-m}(K,M), m=j,\dots ,n$. Hence,
$$
J(rK,sM) = \sum_{m=j}^n r^ms^{n+j-m}\varphi_{m,n+j-m}(K,M)
$$
and, putting $r=s=1$, we get \eqref{ittrans} (for $k=2$).

It remains to prove the properties of the mixed functionals $\varphi_{m,n+j-m}$. The symmetry and the homogeneity property follow for polytopes from \eqref{pol1} and \eqref{pol2}, and for arbitrary bodies by approximation. The valuation property follows from \eqref{ittrans}, if one takes into account the additivity properties of the integrand on the left side and compares these with the different homogeneity properties of the summands on the right side.
\qed
\end{proof}

{\bf Remarks.}
 {\bf 1.} Theorem \ref{trans4} also holds for measure valuations. More precisely, let $\Phi : {\cK}^n\to {\cal M}(\R^n)$ be a continuous, translation covariant valuation and let $\Phi_j$ be its $j$-homogeneous part, $j=0,\dots ,n$ (see \eqref{measval}). Then, for $k\ge 2$, there exist mixed measure-valued functionals $\Phi_{\m}$, $\m\in\mix(j,k)$, on ${({\cK^n})^{k}}$ such that
\begin{align}
  \int_{({\mathbb R}^n)^{k-1}}&
\Phi_j(K_1\cap (K_2+{x_2})\cap \dots \cap (K_k+{x_k}),A_1\cap (A_2+{x_2})\cap \dots \cap (A_k+{x_k}))\nonumber \\
&\times\, \lambda_n^{k-1} (d(x_2,\dots ,x_k))\nonumber \\
&\quad = \sum_{\m\in\mix(j,k)}
\Phi_{\m}(K_1,\dots,K_k;A_{1}\times\cdots\times A_k)\label{ittrans2}
\end{align}
for $K_1,\dots,K_k \in {\cK^n}$ and $A_1,\dots,A_k\in\cB (\R^n)$.

For $(m_1,\ldots,m_k)\in\mix(j,k)$ the mapping $(K_1,...,K_k)\mapsto \Phi_{m_1,\dots,m_k}(K_1,\dots,K_k;A_{1}\times\cdots\times A_k )$ is symmetric (w.r.t. permutations of the indices $1,...,k$), it is homogeneous of degree $m_i$ in $K_i$ and $A_i$, and it is a continuous, translation covariant measure valuation in each of its variables $K_i$.

The proof  follows the same lines as in the case of Theorem \ref{trans4} by starting with the case of polytopes. For $k=2$, we then arrive again at \eqref{transintpol}. In order to extend this expansion to arbitrary bodies $K,M\in\cK^n$, we use the fact that \eqref{transintpol} is equivalent to
\begin{align}\label{transintpol2}
\int_{\R^n}\int_{\R^n}&g(x,x-y)\Phi_j (P\cap (Q+y),dx)\, \lambda_n(dy)\nonumber\\
&= \sum_{m=j}^n
\int_{(\R^n)^2}g(x,y)\Phi_{m,n+j-m}(P,Q;d(x,y))
\end{align}
for each continuous function $g$ on $(\R^n)^2$. Again, the dominated convergence theorem shows that the integral on the left, for arbitrary bodies $K,M$, is a continuous functional $J(g,K,M)$. The homogeneity properties are then used again to show that, for $P_i\to K, Q_i\to M$, each of the integrals
$$
\int_{(\R^n)^2}g(x,y)\Phi_{m,n+j-m}(P_i,Q_i;d(x,y))
$$
on the right side converges. Therefore, the measures $\Phi_{m,n+j-m}(P_i,Q_i;\cdot)$ converge weakly and the limit measures $\Phi_{m,n+j-m}(K,M;\cdot)$ satisfy \eqref{ittrans2} (for $k=2$).

{\bf 2.} Of course, for a measure valuation $\Phi_j$ as above, the total measure $\varphi_j(K)= \Phi_j(K,\R^n)$ satisfies Theorem \ref{trans4} with mixed functionals which are given by the total measures
$$
\varphi_{\m}(K_1,\dots,K_k) = \Phi_{\m}(K_1,\dots,K_k;\R^n\times\cdots\times \R^n) ,\quad \m\in\mix(j,k).
$$
We then say that the measure valuation $\Phi_j$ is a {\it local extension} of the scalar valuation $\varphi_j$ (more generally $\Phi := \sum_{j=0}^n \Phi_j$ is a local extension of $\varphi := \sum_{j=0}^n \varphi_j$). Thus, if a valuation $\varphi\in \Val$ has such a local extension, then the iterated translative formula holds in a global as well as a local version. This fact is of importance, if expectation formulas for non-stationary Boolean models are considered. It is an open question, whether each valuation $\varphi\in \Val$ has a local extension. Local extensions, if they exist, are not unique (corresponding examples are given in \cite{W2015a, W2015b}). It is another open problem to describe all local extensions of a valuation $\varphi$.

{\bf 3.} Due to the summation condition $m_1+\cdots +m_k=(k-1)n+j$, only finitely many different mixed functionals arise in the iterated translative formula \eqref{ittrans}. Namely, the mixed functionals $\varphi_{m_1,\dots,m_k}$ (as well as the local versions $\Phi_{m_1,\dots,m_k}$) split, if one of the parameters $m_i$ equals $n$. In fact, if we consider \eqref{mixtpol} for $m=n$, then $F=P$ (we may assume that $P$ is full dimensional) and $G\in{\cF}_j(Q)$. Then $n(P,Q;F,G) = n(Q,G)$ and $[F,G]=1$, hence
\begin{align*}
\Phi_{n,j}(P,Q;\cdot) &= \lambda_P\otimes\left(\sum_{G\in\cF_{j}(Q)}f_j(n(Q,G))\lambda_G\right)\\
&=\lambda_P\otimes\Phi_j(Q,\cdot).
\end{align*}
This extends to arbitrary bodies $K,M$ (if $\varphi_j$ has a local extension $\Phi_j$) and to the total measures $\varphi_{n,j}(K,M)$. More generally, for $k\geq 2 $ and $(m_1,\ldots,m_{k-1},n)\in\mix(j,k)$ we have
\begin{equation*}
\varphi_{m_1,\dots,m_{k-1},n}(K_1,\dots,K_{k-1},K_k) = \varphi_{m_1,\dots,m_{k-1}}(K_1,\dots,K_{k-1})V_n(K_k)
\end{equation*}
and, if a local extension exists,
\begin{equation*}
\Phi_{m_1,\dots,m_{k-1},n}(K_1,\dots,K_{k-1},K_k;\cdot) = \Phi_{m_1,\dots,m_{k-1}}(K_1,\dots,K_{k-1};\cdot)\otimes \lambda_{K_k} .
\end{equation*}
Because of the symmetry, the case $m_k=n$ also implies corresponding decompositions if $m_i=n, i\in\{1,\dots ,k-1\}$.

\subsection{Kinematic formulas}

The Boolean model $Z$ is isotropic, if and only if the grain distribution $\bf Q$ is rotation invariant. If this is the case, the  translative integrals in \eqref{itint} can be replaced by an integration over the group $G_n$ of rigid motions (with invariant measure $\mu$), hence Theorem \ref{Th2} holds with
\begin{equation}\label{itint2}
\Phi(K_0,K_1,\dots,K_k) := \int_{(G_n)^k}\varphi (K_0\cap g_1K_1\cap\dots\cap g_kK_k) \, \mu^k(d(g_1,\dots ,g_k)) \, .
\end{equation}
Here, Hadwiger's general integral-geometric theorem (see \cite[Theorem 5.1.2]{SW}) shows that
\begin{equation}\label{HadIntTh}
\int_{G_n}\varphi (K\cap gM) \, \mu(dg)  = \sum_{k=0}^n \varphi_{n-k}(K) V_k(M),
\end{equation}
for $K,M\in\cK^n$, where the coefficients $\varphi_{n-k}(K)$ are given by the Crofton-type integrals
\begin{equation}\label{HadIntTh2}
\varphi_{n-k}(K)=\int_{A(n,k)}\varphi (K\cap E) \, \mu_k(dE)
\end{equation}
over the space $A(n,k)$ of affine $k$-dimensional flats in $\R^n$ with invariant measure $\mu_k$. \eqref{HadIntTh} follows from an application of Hadwiger's characterization theorem \cite[Theorem 14]{S2015} and holds for continuous valuations, even without the assumption of translation invariance. The $\varphi_{n-k}$ are then also continuous valuations. If $\varphi\in \Val_j$, then $\varphi_{n-k}\in \Val_{n+j-k}$. If $\varphi\in\Val$ has a local extension $\Phi\ge 0$, then a direct proof of  \eqref{HadIntTh} is possible, based on the translative integral formula \eqref{transintpol} for polytopes and the representation \eqref{mixtpol} (see \cite{W2015b}, for details).

Formula \eqref{HadIntTh} can be easily iterated and yields
\begin{align}\label{HadIntThit}
\int_{(G_n)^k}&\varphi (K_0\cap g_1K_1\cap\dots\cap g_kK_k) \, \mu^k(d(g_1,\dots ,g_k))\\
&= \sum_{\m\in\mix (0,k+1)} c^n_{n-m_0}\varphi_{m_0}(K_0)\prod_{i=1}^k c^{m_i}_nV_{m_i}(K_i) ,
\end{align}
with $\m =(m_0,\dots ,m_k)$ and constants defined by
$$
c^r_s := \frac{r!\kappa_r}{s!\kappa_s} ,
$$
see \cite[Theorem 5.1.4]{SW}.

For $\varphi = V_j$, the integral \eqref{HadIntTh2} can be solved by the Crofton formula and we get
$$
\varphi_{n-k}(K) = c^k_jc^{n+j-k}_n V_{n+j-k}(K)
$$
for $k\ge j$ (and $\varphi_{n-k}(K)=0$ otherwise),
which yields the principal kinematic formula
\begin{equation}\label{pkf}
\int_{G_n}V_j(K\cap gM) \, \mu(dg)  = \sum_{k=j}^n c^k_jc^{n+j-k}_n V_k(K) V_{n+j-k}(M)
\end{equation}
and the iterated version
\begin{align}
&\int_{(G_n)^k}V_j (K_0\cap g_1K_1\cap\dots\cap g_kK_k) \, \mu^k(d(g_1,\dots ,g_k))
\nonumber\\
& =\sum_{\m\in\mix (j,k+1)} c^n_{j}\prod_{i=0}^k c^{m_i}_nV_{m_i}(K_i) .\label{pkfit}
\end{align}

\section{Mean Values for Valuations}
\label{sec:5}

Combining Theorems \ref{Th2} and \ref{trans4}, we obtain the following expectation formula for a stationary Boolean model $Z$ and $\varphi\in \Val_j$,
\begin{align}
\E \varphi (Z\cap K_0) &= \sum_{k=1}^\infty \frac{(-1)^{k-1}}{k!}\gamma^k\sum_{\m\in\mix(j,k+1)}\label{expform}\\
&\int_{\cK^n_c}\cdots \int_{\cK^n_c}
\varphi_{\m}(K_0,K_1,\dots,K_k)\,\bQ(dK_1)\cdots \bQ(dK_k).\nonumber
\end{align}
Here, it is important to observe that the right hand side can be simplified due to the decomposition property which we mentioned above. The resulting formulas are presented in \cite{W2015b}. They depend on the shape and size of the window $K_0$. We can get simpler results, if we eliminate the effect of the window by a suitable limit procedure, namely by normalizing with $V_n(K_0)$ (here we assume $V_n(K_0)>0$) and then letting $K_0$ grow to $\R^n$ (for simplicity, we consider $rK_0, r\to\infty$). Then, on the right hand side all summands with multi-index $\m=(m_0,\ldots,m_k)$ and $m_0<n$ will vanish asymptotically.
If we define for $\m\in\mix(j,k)$ the {\it density} (mean value) $\overline\varphi_{\m}(Y,\dots,Y)$ of the mixed valuation $\varphi_{\m}$ for the (Poisson) particle process $Y$ by
$$
\overline\varphi_{\m}(Y,\dots,Y) = \gamma^k\int_{\cK^n_c}\cdots \int_{\cK^n_c}
\varphi_{\m}(K_1,\dots,K_k)\,\bQ(dK_1)\cdots \bQ(dK_k),
$$
the right hand side thus reads
$$
\sum_{k=1}^\infty \frac{(-1)^{k-1}}{k!}\sum_{\m\in\mix(j,k)}
\overline\varphi_{\m}(Y,\dots,Y).
$$

This also indicates, that the corresponding limit
\begin{equation}\label{limit}
\lim_{r\to\infty}\frac{1}{V_n(rK_0)}\E \varphi (Z\cap rK_0)
\end{equation}
on the left hand side exists. We will discuss this in the following subsection. The second subsection then contains the central result, the explicit expectation formula for valuations and Boolean models. The third subsection shortly discusses the isotropic case.

\subsection{Densities for valuations and random sets}

The following result is Theorem 9.2.1 in \cite{SW}, in a slightly less general form. It shows that the limit in \eqref{limit} exists for valuations $\varphi\in\Val$ (additively extended to the convex ring ${\cal R}^n$) and stationary random sets $Z$ with values in the extended convex ring
$$
{\cal S}^n = \{ F\subset \R^n : F\cap rB^n\in {\cal R}^n \ {\rm for\ all\ }r>0\},
$$
satisfying the condition
\begin{equation}\label{intcondZ}
\E 2^{N(Z\cap B^d)} <\infty.
\end{equation}
Here, $N(A)$, for a set $A\in{\cal R}^n$, is the minimal number $k$ of convex bodies $K_1,\dots ,K_k\in{\cal K}^n$ such that $A=\bigcup_{i=1}^k K_i$.

The class ${\cal S}^n$ consists of countable unions of convex bodies (locally polyconvex sets) and is supplied with the Borel $\sigma$-algebra generated by the Hausdorff metric (which is the same $\sigma$-algebra as the one generated by the {\it hit-or-miss} topology, see \cite[Section 2.4]{SW}). For a stationary Boolean models $Z$ with convex or polyconvex grains, we have $Z(\omega )\in{\cal S}^n$ and \eqref{intcondZ} is satisfied.

\begin{theorem}\label{rs-density}
Let $Z$ be a stationary random set with values in $\cS^n$, satisfying \eqref{intcondZ}. Let $\varphi\in\Val$ and $K\in\cK^n$ with $V_n(K)>0$. Then the limit
$$
\overline\varphi (Z) = \lim_{r\to\infty}\frac{1}{V_n(rK)}\E \varphi (Z\cap rK)
$$
exists and is independent of $K$.
\end{theorem}

For the proof, a functional $\phi \in\Val$ is defined by
$$
\phi (K) = \E \varphi (Z\cap K) ,\quad K\in{\cal K}^n.
$$
Then, Theorem 5 in \cite{S2015} is used to obtain an additive extension of $\phi$ to ro-polyhedra. Since $\R^n$ allows a lattice decomposition into half-open unit cubes $C^d_0,C^d_1,\dots$, one can show directly that
$$
\lim_{r\to\infty}\frac {\phi(rW)}{V_n(rW)} = \phi (C^d_0) .
$$

This argument shows slightly more, namely that
$$
\overline\varphi (Z) =  \E \varphi (Z\cap C^d_0).
$$

We call $\overline\varphi (Z)$ the  $\varphi${\it -density} (or {\it specific $\varphi$-value}) of $Z$. We can estimate the  $\varphi$-density by the $\varphi$-value of $Z$ on the unit cube $C^d$ minus the value $\varphi (Z\cap \partial^+C^d)$ on the upper right boundary $\partial^+C^d$ of $C^d$ (observe that $\partial^+C^d\in{\cal R}^n$).

We get, in particular, the existence of the densities $\overline\varphi_{m,n+j-m}(Z,K)$ for the mixed functionals $\varphi_{m,n+j-m}$ and $K\in{\cal K}^n$.

The following result is a nice application of our translative formula \eqref{ittrans} and generalizes Theorem 9.4.1 in \cite{SW}.
\begin{theorem}\label{rs-density2}
Let $Z$ be a stationary random set with values in $\cS^n$, satisfying \eqref{intcondZ}. Let $\varphi_j\in\Val_j$ and $K\in\cK^n$. Then
$$
\E \varphi_j (Z\cap K) = \sum_{m=j}^n \overline\varphi_{m,n+j-m} (Z,K) .
$$
\end{theorem}

\begin{proof} \smartqed
We can follow the proof of Theorem 9.4.1 in \cite{SW}. We use the stationarity of $Z$ and the translation invariance of $\varphi_j$ to get
\begin{align*}
\E \int_{\R^n} \varphi_j(Z\cap K\cap (rB^n+x))\lambda_n(dx)
&= \E \int_{\R^n} \varphi_j(Z\cap (K+x)\cap rB^n)\lambda_n(dx).
\end{align*}
Now we apply the translative formula to both sides and obtain
$$
\sum_{m=j}^n \E  \varphi_{m,n-m+j}(Z\cap K, rB^n) = \sum_{m=j}^n \E  \varphi_{m,n-m+j}(Z\cap rB^n, K).
$$
We divide both sides by $V_n(rB^n)$ and let $r\to\infty$. On the left, the homogeneity properties of the mixed functionals show that only the summand for $m=j$ remains (and yields  $\E  \varphi_j(Z\cap K)$). Each summand on the right side converges to the corresponding density.
\qed
\end{proof}

The summand on the right hand side for $m=j$ is $\overline\varphi_j (Z) V_n(K)$. Theorem \ref{rs-density2} thus gives the error (or bias), if the mean value $\overline\varphi_j (Z)$ is estimated by the values of $\varphi_j$ for realisations $Z(\omega )$ in a window $K$. As mentioned above, an unbiased estimator of $\overline\varphi_j (Z)$ is given by $\varphi_j (Z(\omega) \cap C^d) - \varphi_j (Z(\omega) \cap \partial^+C^d)$.

\subsection{The mean value formula for Boolean models}

We now come back to Boolean models $Z$ and continue with the formula
$$
\overline\varphi_j (Z)=\sum_{k=1}^\infty \frac{(-1)^{k-1}}{k!}\sum_{\m\in\mix(j,k)}
\overline\varphi_{\m}(Y,\dots,Y),
$$
for $\varphi_j\in\Val_j$, $j\in\{0,\dots ,n\},$ which we have developed. To simplify this formula further, we use again the decomposition property.  For $j=n$, we have only one summand
$$
\overline\varphi_{n,\dots,n}(Y,\dots,Y)= \overline\varphi_n (Y)\overline V_n(Y)^{n-1} = c_n\overline V_n(Y)^{n},$$
which gives us
\begin{align*}
\overline\varphi_n(Z) &=c_n\overline V_n(Z)\\
&=c_n\sum_{k=1}^\infty \frac{(-1)^{k-1}}{k!}\overline V_n(Y)^k\\
&= c_n\left(1-\e^{-\overline V_n(Y)} \right) .
\end{align*}
For $j<n$ and $\m =(m_1,\ldots,m_k)\in\mix(j,k)$, let $s$ be the number of indices which are smaller than $n$. By symmetry we can assume $m_{s+1}=\ldots=m_k=n$. Then
\[
\overline{\varphi}_{m_1,\ldots,m_s,n,\ldots,n}(Y,\ldots,Y)=\overline{\varphi}_{m_1,\ldots,m_s}(Y,\ldots,Y)\overline{V}_n(Y)^{k-s},
\]
where $s\in\{1,\dots ,n-j\}$ and $m_i\in\{ j,\dots , n-1\}$, for $i=1,\dots ,s$.
Introducing the notation
\[
\mix(j):=\{(m_1,\ldots,m_s)\in \{j,\ldots,n-1\}^s \cap \mix(j,s): 1\leq s\leq n-j\},
\]
we obtain
\begin{align*}
&\sum_{\m\in\mix(j,k)}
\overline\varphi_{m}(Y,\dots,Y)= \sum_{s=1}^{(n-j)\wedge k}{k\choose s} \sum_{{\m\in\mix(j)}\atop{|\m|=s}}
\overline\varphi_{m}(Y,\dots,Y)\overline V_n(Y)^{k-s},
\end{align*}
where $(n-j)\wedge k$ denotes the minimum of $n-j$ and $k$. This implies
\begin{align*}
\overline\varphi_j(Z) &= \sum_{k=1}^\infty (-1)^{k-1} \sum_{s=1}^{(n-j)\wedge k}\frac{1}{(k-s)!}\overline V_n(Y)^{k-s}
\sum_{{\m\in\mix(j)}\atop{|\m|=s}}\frac{1}{s!}
\overline\varphi_{\m}(Y,\dots,Y)\\
&= \sum_{s=1}^{n-j}\sum_{r=0}^\infty\frac{(-1)^{r+s-1}}{r!}\overline V_n(Y)^{r} \sum_{{\m\in\mix(j)}\atop{|\m|=s}}\frac{1}{s!}
\overline\varphi_{\m}(Y,\dots,Y)\\
&=\e^{-\overline V_n(Y)} \sum_{\m\in\mix(j)}\frac{(-1)^{|\m|-1}}{|\m|!}
\overline\varphi_{\m}(Y,\dots,Y).
\end{align*}
For $j=n-1$, the sum in the formula above reduces to $\overline\varphi_{n-1}(Y)$. Hence we have obtained the following result.

\begin{theorem}\label{bm}
Let $Z$ be a stationary Boolean model with convex grains and let $\varphi_j\in\Val_j$. Then,
$$
\overline\varphi_n(Z) = c_n\left(1-\e^{-\overline V_n(Y)} \right),
$$
$$
\overline\varphi_{n-1}(Z) = \e^{-\overline V_n(Y)}\overline\varphi_{n-1}(Y) ,
$$
and
\begin{align}\label{BM-j}
\overline\varphi_j(Z)
&=\e^{-\overline V_n(Y)} \sum_{\m\in\mix(j)}\frac{(-1)^{|\m|-1}}{|\m|!}
\overline\varphi_{\m}(Y,\dots,Y),
\end{align}
for $j=0,\dots ,n-2$.
\end{theorem}

In \eqref{BM-j}, the multi-index $\m\in\mix(j)$ with $|\m |=1$ is $\m =(j)$, it yields the summand $\overline\varphi_{j}(Y)$. The remaining summands have multi-indices $\m =(m_1,...,m_s)$ with $s>1$ and $m_i\in \{j+1,\dots ,n-1\}$, due to the definition of $\mix(j)$ and the summation rule in $\mix(j,s)$.

\subsection{The isotropic case} \label{Sec:IsotropicCase}

If the Boolean model $Z$ is stationary and isotropic, we can obtain a result analogous to Theorem \ref{bm} by using the iterated kinematic formula \eqref{HadIntThit}. Equivalently, we can use the rotation invariance of $\bf Q$ to show that the mean value $\overline\varphi_{\m}(Y,\dots,Y)$ in \eqref{BM-j} satisfies
$$
\overline\varphi_{\m}(Y,\dots,Y) = c_{\m}(\varphi_j) \prod_{i=1}^s \overline V_{m_i}(Y),
$$
for $\m = (m_1,\dots ,m_s)$ with constants $c_{\m}(\varphi_j)$ depending on $\varphi_j$.

In the case $\varphi_j = V_j$, we have
$$
c_{\m}(V_j) = c^n_j\prod_{i=1}^s c^{m_i}_n .
$$
Using this in Theorem \ref{bm}, we see that, in the isotropic case, all densities $\overline{V}_j(Z)$ can be expressed by the densities $\overline{V}_j(Y)$ and we obtain the famous Miles formulas
$$
\overline V_n(Z) = c_n\left(1-\e^{-\overline V_n(Y)} \right),
$$
$$
\overline V_{n-1}(Z) = \e^{-\overline V_n(Y)}\overline\varphi_{n-1}(Y) ,
$$
and
\begin{align*}
\overline V_j(Z)
&=\e^{-\overline V_n(Y)} \sum_{\m\in\mix(j)}\frac{(-1)^{|\m|-1}}{|\m|!}
c_j^n \prod_{i=1}^s c_n^{m_i}\overline V_{m_i}(Y),
\end{align*}
for $j=0,\dots ,n-2$.

These density formulas can be inverted successively from top to bottom. Since, for convex grains, the mean value $\overline V_0(Y)$ equals the intensity $\gamma$, we obtain in this way an expression for $\gamma$ in terms of the densities $\overline V_j(Z)$ of the Boolean model $Z$.

\section{Special Cases}

We now discuss which formulas arise from Theorem \ref{trans4} and Theorem \ref{bm}, if special valuations $\varphi$ are considered.

\subsection{Mixed volumes}
As a first case, we consider the {\it mixed volume} $\varphi (K) = V(K[j],M_{j+1},\dots,M_n)$, for fixed bodies $M_{j+1},\dots,M_n\in\cK$. It follows from the properties of the intrinsic volume $V_j$ (which corresponds to the case $M_{j+1}= \dots =M_n=B^n$) that $\varphi$ is in $\Val_j$. Formula \eqref{ittrans} thus gives
\begin{align}
&  \int_{({\mathbb R}^n)^{k-1}}
V(K_1\cap (K_2+x_2)\cap \dots \cap (K_k+x_k)[j], M_{j+1},\dots,M_n)\, \lambda^{k-1} (d(x_2,\dots ,x_k))\nonumber \\
& = \sum_{\m\in\mix(j,k)}
V_{\m}(K_1,\dots,K_k;M_{j+1},\dots,M_n),\label{mixedvol}
\end{align}
with mixed functionals $V_{\m}(K_1,\dots,K_k;M_{j+1},\dots,M_n)$.
The special case $M_{j+1}=\dots =M_n=B^n$, yields the iterated translative formula for intrinsic volumes,
\begin{align}
&  \int_{({\mathbb R}^n)^{k-1}}
V_j(K_1\cap (K_2+x_2)\cap \dots \cap (K_k+x_k))\, \lambda^{k-1} (d(x_2,\dots ,x_k))\nonumber \\
& = \sum_{\m\in\mix(j,k)}
V_{\m}(K_1,\dots,K_k).\label{mixedvol2}
\end{align}
(Note, however, that $V_j(K)$ and $V(K[j],B^n,\dots,B^n)$ differ by a constant. Therefore, $V_{\m}(K_1,\dots,K_k)$ and $ V_{\m}(K_1,\dots,K_k; B^n,\dots ,B^n)$ also differ by the same constant.)

We also remark that $K\mapsto V(K[j],M_{j+1},\dots,M_n)$ has a local extension given, up to a constant, by the mixed curvature measure $C(K[j],M_{j+1},\dots,M_n;\cdot)$ introduced and studied in \cite{KW} (see also \cite{Hug} and \cite{HL}). This implies a corresponding local integral formula coming from \eqref{ittrans2} which we do not copy here. Instead, we emphasize the special case $M_{j+1}=\dots =M_n=B^n$, where we have a multiple of the $j$th order curvature measure $C_j(K,\cdot), j=0,\dots ,n$, and where we obtain the iterated translative formula

\begin{align}
&  \int_{({\mathbb R}^n)^{k-1}}
C_j(K_1\cap (K_2+x_2)\cap \dots \cap (K_k+x_k),A_1\cap (A_2+x_2)\cap \dots\nonumber\\
&\quad\quad\ldots \cap (A_k+x_k))\, \lambda^{k-1} (d(x_2,\dots ,x_k))\nonumber \\
& = \sum_{\m\in\mix(j,k)}
C_{\m}(K_1,\dots,K_k;A_1\times\cdots\times A_k),\label{mixedmeas}
\end{align}
with mixed measures, which are different from the mixed curvature measures mentioned above. For example, the mixed curvature measure $C(M_1,\dots, M_n;\cdot)$ is a measure on $\R^n$, depending on $n$ bodies and with total degree of homogeneity $n$. In contrast to this, the mixed measure $C_{m_1,\dots ,m_k}(K_1,\dots,K_k;\cdot)$ is a measure on $(\R^n)^k$, depends on $k$ bodies and has total degree of homogeneity $m_1+\dots + m_k = (k-1)n+j$. Of course, these homogeneity properties already arise on the global level and distinguish mixed volumes and mixed functionals from translative integral geometry. There is a special case where the two series of functionals meet, for $j=0$ (where we have the translative formula for the Euler characteristic). Here,
$$
V_{m,n-m}(K,M) = {n\choose m}V(K[m],-M[n-m]),
$$
for $m=0,\dots ,n$.

Theorem \ref{bm} implies mean value formulas for mixed volumes and Boolean models. We only state the result for the intrinsic volumes, which is Theorem 9.1.5 in \cite{SW} and reads
$$
\overline V_n(Z) =1-\e^{-\overline V_n(Y)},
$$
$$
\overline V_{n-1}(Z) = \e^{-\overline V_n(Y)}\overline V_{n-1}(Y) ,
$$
and
\begin{align}
\overline V_j(Z)
&=\e^{-\overline V_n(Y)} \sum_{\m\in\mix(j)} \frac{(-1)^{|\m|-1}}{|\m|!}
\overline V_{\m}(Y,\dots,Y),\label{DensityFormIntVol}
\end{align}
for $j=0,\dots ,n-2$.

If $Z$ is isotropic, then the density of the mixed functional $V_{m_1,\ldots,m_s}$ splits,
\begin{align}
 \overline V_{m_1,\dots,m_s}(Y,\dots,Y) = c^n_j\prod_{i=1}^s c^{m_i}_n\overline V_{m_i}(Y),\label{SplitMixDensIsotropic}
\end{align}
with $$ c^k_i:=\frac{k!\kappa_k}{i!\kappa_i}, \quad i,k\in\N_0.
$$
This is Theorem 9.1.4 in \cite{SW} (with corrected constants).

\subsection{Support functions}

As a next case, we consider the {\it (centered) support function} $\varphi (K) = h^\ast (K,\cdot )$. This is a translation invariant, continuous and additive functional, which is homogeneous of degree 1, with values in the Banach space of centered continuous functions on $\s^{n-1}$. To fit this case into our framework, we may apply the results for scalar valuations point-wise, that is, for $h^\ast (K,u), u\in \s^{n-1}$. The iterated translative formula then reads
\begin{align}
&  \int_{({\mathbb R}^n)^{k-1}}
h^\ast (K_1\cap (K_2+x_2)\cap \dots \cap (K_k+x_k),\cdot)\, \lambda^{k-1}_n (d(x_2,\dots ,x_k))\nonumber \\
& = \sum_{\m\in\mix(1,k)}
h^\ast_{\m}(K_1,\dots,K_k,\cdot),\label{suppf}
\end{align}
with mixed support functions $h^\ast_{m_1,\dots,m_k}(K_1,\dots,K_k,\cdot)$, $(m_1,\ldots,m_k)\in\mix(1,k)$. This integral formula was studied in \cite{W95} and \cite{GW03}. In the latter paper, it was also shown that, for $k=2$, the mixed function $h^\ast_{m,n+1-m}(K_1,K_2,\cdot)$ is indeed a support function. For general $k$ this was shown, with a different proof, by Schneider (\cite{Sch03}).

The formula for Boolean models $Z$ reads
\begin{align*}
\overline h^\ast(Z,\cdot)
&=\e^{-\overline V_n(Y)} \sum_{\m\in\mix(1)} \frac{(-1)^{|\m|-1}
}{|\m|!}
\overline h^\ast_{\m}(Y,\dots,Y,\cdot).
\end{align*}

Again, there is a local extension of $K\mapsto h^\ast(K,u)$ given by the mixed measure $\phi_{1,n-1}(K,u^+;\cdot)$ where $u^+$ is the closed half-space with outer normal $u$ (see \cite{W95, GW03}). The corresponding iterated translative formula for this mixed measure is a consequence of Theorem \ref{trans4}, but it also follows from the general results in \cite[Section 6.4]{SW}.

\subsection{Area measures}
\label{sec:6.3}
Next, we consider the {\it area measure} map $S_j : K\mapsto S_j(K,\cdot)$. It is a translation invariant additive and measure-valued functional which is continuous with respect to the weak topology of measures.  To fit these measure-valued notions into our results, we cannot consider them point-wise, for a given Borel set, since this would not yield a continuous valuation. However, we can apply our results to the integral
$$\varphi^f_j (K) = \int_{\s^{n-1}} f(u) S_j(K,du)$$
with a continuous  function $f$ on $\s^{n-1}$. Namely, $\varphi_j^f$ is an element of $\Val$ and fulfills, by Theorem \ref{trans4}, the iterated translative formula
\begin{align*}
& \int\limits_{\left(\mathbb{R}^n\right)^{k-1}} \varphi^f_j(K_1\cap(K_2+x_2)\cap\ldots\cap(K_k+x_k)) \lambda_n^{k-1}(d(x_2,\ldots,x_k))\\
& \quad = \sum\limits_{\m\in\mix(j,k)}\varphi^f_{\m}(K_1,\ldots,K_k)
\end{align*}
with unique mixed functionals $\varphi^f_{m_1,\ldots,m_k}$, $(m_1,\ldots,m_k)\in\mix(j,k)$.
Let $C(\s^{n-1})$ denote the space of continuous functions on the sphere.
For $\m\in\mix(j,k)$ the mapping $f\mapsto \varphi^f_m(K)$ is a continuous linear functional on $C(\s^{n-1})$.
The Riesz representation theorem therefore implies the existence of a unique finite (signed) measure $S_{\m}(K_1,\ldots,K_k;\cdot)$ on $\s^{n-1}$ with
\[
\varphi^f_{\m}(K_1,\ldots,K_k)=\int\limits_{\s^{n-1}}f(u)S_{\m}(K_1,\ldots,K_k;du),
\]
which we call the mixed measure of area type.
Therefore, we obtain a translative formula for area measures reading
\begin{align*}
& \int\limits_{\left(\mathbb{R}^n\right)^{k-1}} S_j(K_1\cap(K_2+x_2)\cap\ldots\cap(K_k+x_k),\cdot) \lambda_n^{k-1}(d(x_2,\ldots,x_k))\\
& \quad = \sum\limits_{\m\in\mix(j,k)}S_{\m}(K_1,\ldots,K_k;\dot).
\end{align*}
Since area measures have centroid $0$, the same is true for the mixed measures. We emphasize again the difference between the mixed measures of area type, which arise in the translative formula, and the mixed area measures $S(K_1,\dots, K_{n-1},\cdot)$ which are defined as a coefficient in the multilinear expansion of $S_{n-1}(\alpha_1K_1+\dots +\alpha_{n-1}K_{n-1},\cdot), \alpha_i\ge 0$. Both types of measures are measures on the unit sphere but they depend on different numbers of bodies and have different homogeneity properties.

The translative formula for area measures was originally obtained in \cite{Hug}.

We remark that the mixed area-type measures for $j=0$ are trivial. Since $S_0(K,\cdot)=V_0(K) \sigma(\cdot)$ where $\sigma$ is the spherical Lebesgue measure we have by the uniqueness of the mixed measures
\begin{align*}
S_{\m}(K_1,\ldots,K_k;\cdot)=
V_{\m}(K_1,\ldots,K_k)\,\sigma(\cdot), \quad \m\in\mix(k,0),k\geq 2.
\end{align*}
The translative formulas imply formulas for Boolean models (see \cite[Corollary 4.1.4]{Hoerr}) of the form
\begin{align}\label{bm-areameas}
\overline{S}_j(Z,\cdot)=\mathrm{e}^{- \overline{V}_n(X)}\sum\limits_{\m\in\mix(j)}\frac{(-1)^{|\m|-1}}{|\m|!}
\overline{S}_{\m}(Y,\ldots,Y;\cdot).
\end{align}
This result follows now also from Theorem \ref{bm} using the functional analytic approach described above.

If $Z$ is an isotropic Boolean model, the measure $\overline{S}_j(Z,\cdot)$ is rotation invariant. Since the spherical Lebesgue measure $\sigma$ is up to normalization the unique measure on $\s^{n-1}$ with this property, we have
\begin{align*}
 \overline{S}_j(Z,\cdot)
&= \frac{n\kappa_{n-j}}{\omega_n\binom{n}{j}}\;\overline{V}_j(Z)\,\sigma
\end{align*}
and the formula \eqref{bm-areameas} is equivalent to the corresponding result for the specific intrinsic volume $\overline{V}_j(Z)$ (this also implies a rotation formula for mixed area-type measures).

The local extension of the valuation $K\mapsto S_j(K;C), C\subset \s^{n-1},$ is (up to a constant) given by $K\mapsto\Lambda_j(K;\cdot\times C)$, where $\Lambda_j(K,\dot)$ is the support measure of $K$ introduced in Section \ref{sec:2}.
For the support measures a local translative integral formula similar to \eqref{mixedmeas} holds wich was originally shown in \cite{Hug}.

\subsection{Flag measures}

Now, we use the functional analytic approach just described in a similar situation, for {\it flag measures} of convex bodies.
A general reference for flag measures is the overview article \cite{HTW}. The flag measure we consider in the following are a version of the translation invariant flag area measures which are considered in \cite{GHHRW} and related to the flag area measures in \cite{HTW} via \cite[(2.1)]{GHHRW} and a renormalization.
We first describe the underlying notions concerning flag manifolds. Recall that $G(n,j)$ denotes the Grassmannian of $j$-dimensional subspaces (which we supply with the invariant probability measure $\nu_j$) and define corresponding flag manifolds by
$$
F(n,j) = \{(u,L) : L\in G(n,j), u\in L\cap \s^{n-1}\}
$$
and
$$
F^\perp (n,j) = \{ (u,L) : L\in G(n,j), u\in L^\perp\cap \s^{n-1}\} .
$$
Both flag manifolds carry natural topologies (and invariant Borel probability measures) and $F(n,n-j)$ and $F^\perp(n,j)$ are homeomorphic via the orthogonality map $\rho : (u,L)\mapsto (u, L^\bot)$.  We define a flag measure $\psi_j (K,\cdot)$ as a projection mean of area measures,
\begin{equation}\label{flagprojformula}
\psi_j (K,A) = \int_{G(n,j+1)}\int_{\s^{n-1}\cap L} {\bf 1}\{ (u,L^\bot\vee u)\in A\} S'_j(K|L,du) \nu_{j+1}(dL)
\end{equation}
for a Borel set $A\subset F(n,n-j)$, where $L^\bot\vee u$ is the subspace generated by  $L^\bot$ and the unit vector $u$ and where the prime indicates the area measure calculated in the subspace $L$ (for the necessary measurability properties needed here and in the following, we refer to \cite{Hind}). Using the homeomorphism $\rho$, we can replace $\psi_j (K,\cdot)$ by a measure $\psi_j^\bot (K,\cdot)$ on $F^\perp(n,j)$ given by
\begin{equation}\label{flagprojformula2}
\psi_j^\bot (K,A) = \int_{G(n,j+1)}\int_{\s^{n-1}\cap L} {\bf 1}\{ (u,L\cap u^\bot)\in A\} S'_j(K|L,du) \nu_{j+1}(dL) .
\end{equation}
These two (equivalent) versions of the same flag measure are motivated by the fact that their images under the map $(u,L)\mapsto u$ are in both cases the $j$th order area measure  $S_j(K,\cdot)$. Both measures, $\psi_j (K,\cdot)$ and $\psi_j^\bot (K,\cdot)$ have a local version $\lambda_j(K,\cdot)$, respectively $\lambda_j^\bot (K,\cdot)$, which is obtained by replacing in \eqref{flagprojformula} and \eqref{flagprojformula2} the area measure $S'_j(K|L,\cdot)$ by a multiple of the support measure $\Lambda'_j(K|L,\cdot)$ (see \cite[Theorem 4]{HTW}). In the following, we concentrate on $\psi_j (K,\cdot)$, formulas for the other representation $\psi_j^\bot (K,\cdot)$ follow in a similar way.

The measure $\psi_j (K,\cdot)$ is centered in the first component,
$$
\int_{F(n,n-j)} u \psi_j(K,d(u,L)) = 0,
$$
as follows from the corresponding property of area measures. Let $C(F(n,n-j))$ be the Banach space of continuous functions on $F(n,n-j)$ and choose $f\in C(F(n,n-j))$. Then,
$$\varphi^f_j : K\mapsto \int_{F(n,n-j)} f(u,L) \psi_j(K,d(u,L))$$
is in $\Val_j$. Consequently, we obtain the iterated translative formula
\begin{align}
  \int_{({\mathbb R}^n)^{k-1}}
\varphi^f_j(K_1\cap &(K_2+x_2)\cap \dots \cap (K_k+x_k))\, \lambda^{k-1} (d(x_2,\dots ,x_k))\nonumber \\
& = \sum_{\m\in\mix(j,k)}
\varphi_{\m}^f(K_1,\dots,K_k),\label{flagmeas}
\end{align}
with mixed functionals $\varphi_{\m}^f(K_1,\dots,K_k)$, $\m\in\mix(j,k)$. For fixed bodies $K_1,\dots , K_k$, the left side is a continuous linear functional on $C(F(n,n-j))$, if we let $f$ vary. Namely,
$$f\mapsto \varphi_j^f(K_1\cap (K_2+x_2)\cap \dots \cap (K_k+x_k))$$
is continuous and linear, for each $x_1,\dots ,x_k$, and this carries over to the integral. Replacing $K_1,\dots ,K_k$ by $\alpha_1K_1,\dots ,\alpha_kK_k, \alpha_i>0,$ we use the homogeneity properties of $\varphi^f_{m_1,\dots,m_k}$ to see that the right side is a polynomial in $\alpha_1,\dots ,\alpha_k$. This shows that the coefficients $\varphi^f_{m_1,\dots,m_k}(K_1,\dots,K_k)$ of this polynomial must be continuous linear functionals on $C(F(n,n-j))$, too. By the Riesz representation theorem we obtain unique finite (signed) measures $\psi_{m_1,\dots,m_k}(K_1,\dots,K_k;\cdot)$ on $F(n,n-j)$ such that
$$
\varphi^f_{m_1,\dots,m_k}(K_1,\dots,K_k) = \int_{F(n,n-j))} f(u,L) \psi_{m_1,\dots,m_k}(K_1,\dots,K_k;d(u,L))
$$
for all $f\in C(F(n,n-j))$.  We call them the {\it mixed flag measures}. They are again centered in the first component.

Hence we obtain the iterated translative formula for flag measures,
\begin{align}
  \int_{({\mathbb R}^n)^{k-1}}&
\psi_j(K_1\cap (K_2+x_2)\cap \dots \cap (K_k+x_k),\cdot)\, \lambda^{k-1} (d(x_2,\dots ,x_k))\nonumber \\
& = \sum_{\m\in\mix(j,k)}
\psi_{\m}(K_1,\dots,K_k;\cdot).\label{flagmeas2}
\end{align}

A mean value formula for flag measures of Boolean models $Z$ follows from Theorem \ref{bm}. Since it looks very similar to \eqref{bm-areameas} (but is in fact a generalization), we do not copy it here.

\section{Tensor Valuations and Boolean Models}
\label{sec:6}
Finally, we consider the \textit{Minkowski tensors} $K\mapsto \Phi_j^{r,s}(K)$ which are the central objects of various chapters of this volume.
They are defined as integrals with respect to the support measures. Therefore, Section \ref{sec:6.3} implies the iterated translative formula
\begin{align*}
& \int\limits_{\left(\mathbb{R}^n\right)^{k-1}} \Phi_j^{r,s}(K_1\cap(K_2+x_2)\cap\ldots\cap(K_k+x_k)) \lambda_n^{k-1}(d(x_2,\ldots,x_k))\\
& \quad = \sum\limits_{\m\in\mix(j,k)}\Phi^{r,s}_{\m}(K_1,\ldots,K_k),
\end{align*}
with mixed tensor valuations $(K_1,\ldots,K_k)\mapsto \Phi_{\m}^{r,s}(K_1,\ldots,K_k)$, $\m\in\mix(j,k)$. The mixed tensor valuations $\Phi_{m_1,\ldots,m_k}^{r,s}$ are homogeneous of order $m_1+r$ with respect to the first argument $K_1$ and homogeneous of order $m_i$ with respect to $K_i$ for $i\geq 2$. For $r=0$ the Minkowski tensors are translation invariant. In this case their coordinates are elements of $\Val$ and as an alternative to the above approach via support measures, Theorem \ref{trans4} can be applied directly.

It is convenient to define \textit{local Minkowski tensors} as the tensor-valued signed measures on $\R^n$ given by
\[
\Phi_j^{r,s}(K,A):=c^{r,s}_{n-j}\int_{A\times \s^{n-1}} x^ru^s \Lambda_j(K,d(x,u))
\]
for Borel sets $A\subset \R^n$.
They fulfill the translative formula
\begin{align}
&  \int_{({\mathbb R}^n)^{k-1}}
\Phi^{r,s}_j(K_1\cap (K_2+x_2)\cap \dots \cap (K_k+x_k),A_1\cap (A_2+x_2)\cap \dots\nonumber\\
&\quad\quad\ldots \cap (A_k+x_k))\, \lambda^{k-1} (d(x_2,\dots ,x_k))\nonumber \\
&\quad = \sum_{\m\in\mix(j,k)}
\Phi^{r,s}_{\m}(K_1,\dots,K_k;A_1\times\cdots\times A_k)\label{mixedtensormeas}
\end{align}
with mixed local Minkowski tensors $\Phi_{\m}^{r,s}(K;\cdot)$, $\m\in\mix(j,k)$.
For the translation invariant Minkowski tensors we obtain also density formulas for Boolean models reading
\begin{align*}
\overline{\Phi}^{\,0,s}_j(Z)=\mathrm{e}^{- \overline{V}_n(X)}\sum\limits_{\m\in\mix(j)}\frac{(-1)^{|\m|-1}}{|\m|!}
\overline{\Phi}^{\,0,s}_{\m}(Y,\ldots,Y).
\end{align*}
If $Z$ is an isotropic Boolean model we have
\[ \overline{\Phi}_j^{\,0,s}(Z)=\mathbf{1}\{s\in 2\N_0\}\alpha_{n,j,s} Q^{\frac{s}{2}} \overline{V}_j(Z),
\]
where
\[\alpha_{n,j,s}:=\frac{2}{s!}\,
\frac{\omega_{n-j}\,\omega_{s+n}}{\omega_n\,\omega_{n-j+s}\,\omega_{s+1}}.
\]
The translative formula and the density formulas for Minkowski tensors are contained in \cite{HHKM}. In \cite{HHKM} also mean value formulas for all Minkowski tensors and a parametric example illustrating the usefulness of the Minkowski tensors for the study of a special non-isotropic Boolean model can be found.

\section{Concluding Remarks and Outlook}
\label{sec:7}

As we have mentioned before, the methods and results for the use of valuations with stationary Boolean models $Z$ can be extended to  non-stationary $Z$ under mild regularity assumptions. We describe this situation in the following, but leave out many details for which we refer to the literature.

We recall from Proposition \ref{3.1.1} the definition of the measure $\Theta$ and the Poisson property of $Y$, which shows that $\Theta$ is a translation invariant measure on $\cK^n$, which satifies
$$
\E Y = \Theta .
$$
Since therefore $\Theta (A)$, for a Borel set $A\subset \cK^n$, describes the mean number of particles from $Y$ which fall into $A$, $\Theta$ is called the {\it intensity measure} of $Y$. For the discussion of the non-stationary case, we allow Poisson particle processes $Y$ on $\cK^n$, where the intensity measure $\Theta := \E Y$ on $\cK^n$ is no longer translation invariant, but is absolutely continuous with respect to a translation invariant measure. We call such a measure {\it translation regular}. It then follows that
$$
\Theta (A) = \int_{\cK^n_c}\int_{\R^n} {\bf 1}_A(K+x)\eta (K,x)\lambda_n(dx){\bf Q}(dK)
$$
for some probability measure ${\bf Q}$ on $\cK^n_c$ and a measurable function $\eta\ge 0$ on ${\cK^n_c}\times\R^n$ (see \cite[(11.1)]{SW}). In general, $\bf Q$ and $\eta$ are not uniquely determined by $\Theta$, but they are if $\eta$ does not depend on $K$, hence can be considered as a function on $\R^n$ alone. We will assume this throughout the following and refer to \cite[Section 11.1]{SW} and \cite{W2015b}, for the more general situation. Then, $\eta$ is called the {\it intensity function} and $\bf Q$ the {\it distribution of the typical grain} of the Poisson particle process $Y$. The interpretation is similar to the stationary case. Points in space are distributed according to the intensity function $\eta$ (by a Poisson process $X_0$ in $\R^n$ with intensity measure $\int \eta d\lambda_n$). Then convex bodies are attached to the points independently and with distribution $\bf Q$.

Let now $\varphi\in\Val$ have a local extension $\Phi$. Then $\Phi (Z,\cdot)$ is a signed Radon measure (defined on bounded Borel sets of $\R^n$) which is absolutely continuous to the Lebesgue measure $\lambda_n$. We denote its (almost everywhere existing) density by $\overline\varphi (Z,\cdot)$ (this is a measurable function on $\R^n$). Then we have, as a generalization of Theorem \ref{Th2},
\begin{align*}
\overline\varphi (Z,z)
 = \sum_{k=1}^\infty \frac{(-1)^{k-1}}{k!}&\int_{\cK^n_c}\cdots \int_{\cK^n_c} \int_{(\R^n)^k}\eta(z-x_1)\cdots\eta (z-x_k)\\
&\quad\times \Phi_{(k)}(K_1,\dots,K_k;d(x_1,\dots ,x_k))\,\bQ(dK_1)\cdots \bQ(dK_k)
\end{align*}
where the measure $\Phi_{(k)}(K_1,\dots,K_k;\cdot)$ is given by
\begin{align}\label{itint-gen}
&\Phi_{(k)}(K_1,\dots,K_k;A_1\times\dots\times A_k)\nonumber\\
&\ := \int_{(\R^n)^{k-1}}\Phi (K_1\cap (K_2+x_2)\cap\dots\cap(K_k+x_k),A_1\cap (A_2+x_2)\cap\dots\cap(A_k+x_k))\nonumber \\
&\quad\quad \times \lambda_n^{k-1}(d(x_2,\dots ,x_k)) \, ,
\end{align}
for Borel sets $A_1,\dots ,A_k\subset \R^n$. It is remarkable that, in this non-stationary situation, still an iterated translative integral shows up. Using \eqref{ittrans2}, we can now proceed as in Section 5.2 and obtain
$$
\overline\varphi_n(Z,z) = c_n\left(1-\e^{-\overline V_n(Y,z)} \right),
$$
$$
\overline\varphi_{n-1}(Z,z) = \e^{-\overline V_n(Y,z)}\overline\varphi_{n-1}(Y,z) ,
$$
and
\begin{align}\label{BM-j-inst}
\overline\varphi_j(Z,z)
&=\e^{-\overline V_n(Y,z)} \sum_{\m\in\mix(j)}\frac{(-1)^{|\m|-1}}{|\m|!}
\overline\varphi_{\m}(Y,\dots,Y,z,\dots ,z),
\end{align}
for $j=0,\dots ,n-2$ and $\lambda_n$-almost all $z\in\R^n$. Here, the mean values for $Y$ are defined by
$$
\overline V_n(Y,z) =\int_{\cK^n_c}\int_{\R^n} \eta (z-x) \lambda_n(dx){\bf Q}(dK)
$$
and
\begin{align*}
\overline \varphi_{\m}(Y,\dots,Y,z,\dots ,z) =&\int_{\cK^n_c}\cdots \int_{\cK^n_c}\int_{(\R^n)^k} \eta (z-x_1)\cdots \eta (z-x_k) \\
&\quad\times \Phi_{\m}(K_1,\dots,K_k; d((x_1,\dots ,x_k)){\bf Q}(dK_1)\cdots {\bf Q}(dK_k),
\end{align*}
see \cite[Theorem 6.2]{W2015b}.

Specializing to the examples discussed in Section 6, we obtain from \eqref{BM-j-inst}  formulas for various  geometric mean values for general Boolean models.
In particular, for the translation invariant local Minkowski tensors we obtain the formulas
$$
\overline\Phi_{n-1}^{\,0,s}(Z,z) = \e^{-\overline V_n(Y,z)}\overline\Phi^{\,0,s}_{n-1}(Y,z) ,
$$
and
\begin{align}
\overline\Phi^{\,0,s}_j(Z,z)
&=\e^{-\overline V_n(Y,z)} \sum_{\m\in\mix(j)}\frac{(-1)^{|\m|-1}}{|\m|!}
\overline\Phi^{\,0,s}_{\m}(Y,\dots,Y,z,\dots ,z),
\end{align}
for $j=0,\dots ,n-2$, $s\in\N_0$ and $\lambda_n$-almost all $z\in\R^n$.

We conclude this article with an outlook on the recent development of applying harmonic intrinsic volumes in the study of stationary non-isotropic Boolean models.
Harmonic intrinsic volumes are integrals of spherical polynomials with respect to the area measures $S_j(K;\cdot) =c_{n,j}\Lambda_j(K,\R^d\times \cdot)$, $j\in\{0,\ldots,n-1\}$.
Let $\mathbf{S}_l$ denote the space of spherical harmonics (i.e. homogeneous spherical polynomials $p$ with $\Delta p=0$) of degree $l$ and let $D(n,l)$ be the dimension of $\mathbf{S}_l$. Let $Y_{l,1},\ldots,Y_{l,D(n,l)}$ be an orthonormal basis of $\mathbf{S}_l$ with respect to the $L^2$-scalar product with the measure $\omega_n^{-1}\sigma$. Then, \textit{harmonic intrinsic volumes} are defined by
\[
V_j^{l,p}(K):=c_{n,j}\int\limits_{S^{n-1}}Y_{l,p}(u)S_j(K,du),
\]
where
\[
c_{n,j}:=\binom{n}{j}\frac{1}{n\kappa_{n-j}}.
\]
The harmonic intrinsic volumes $V_j^{l,p}$ are elements of $\Val$. Furthermore it holds
\[V_j^{0,1}=V_j,\]
i.e. the usual intrinsic volumes are contained in the collection of harmonic intrinsic volumes.
They fulfill an interesting rotation formula
\[
\int\limits_{SO_d}V_j^{l,p}(\vartheta K)\nu(d\vartheta)=
\begin{cases}
V_j(K),& (l,p)=(0,1),\\
0,& \text{otherwise}.
\end{cases}
\]

Since the harmonic intrinsic volumes are integrals with respect to the area measures, Section \ref{sec:6.3} implies the iterated translative formula,
\begin{align*}
& \int\limits_{\left(\mathbb{R}^n\right)^{k-1}} V_j^{l,p}(K_1\cap(K_2+x_2)\cap\ldots\cap(K_k+x_k)) \lambda_n^{k-1}(d(x_2,\ldots,x_k))\\
& \quad = \sum\limits_{\m\in\mix(j,k)}V^{l,p}_{\m}(K_1,\ldots,K_k).
\end{align*}
Consequently, also density formulas for Boolean models are obtained reading
\begin{align*}
\overline{V}^{\,l,p}_j(Z)=\mathrm{e}^{- \overline{V}_n(X)}\sum\limits_{\m\in\mix(j)}\frac{(-1)^{|\m|-1}}{|\m|!}
\overline{V}^{\,l,p}_{\m}(Y,\ldots,Y).
\end{align*}
If $Z$ is an isotropic Boolean model, we have
\[ \overline{V}_j^{\,l,p}(Z)=
\begin{cases} \overline{V}_j(Z),& (l,p)=(0,1),\\
0,& \text{otherwise},
\end{cases}
\]
a property which already indicates that the harmonic intrinsic volumes are particularly useful for non-isotropic Boolean models. This turns out to be true, if we consider a Boolean model where the grain distribution $\bf Q$ is \textit{rotation regular}, i.e. satisfies
$$
\mathbf{Q} (A) = \int_{\cK^n_c}\int_{SO_n} {\bf 1}_A(\vartheta K)\eta (K,\vartheta)\nu(d\vartheta){\tilde{\bf Q}}(dK)
$$
for some rotation invariant probability measure ${\tilde{\bf Q}}$ on $\cK^n_c$ and a measurable function $\eta\ge 0$ on ${\cK^n_c}\times SO_n$. The measure $\tilde{\bf Q}$ is unique and the function $\eta$ is unique $\tilde{\bf Q}\otimes\nu$-everywhere under the additional assumptions
\[\int\limits_{SO_n}\eta(K,\vartheta)\nu(d\vartheta)=1 \quad\text{ and }\quad
\eta(\sigma K,\vartheta)=\eta(K,\vartheta\sigma),
\]
for $K\in \cK^n_c,\vartheta,\sigma\in SO_n.$
It was recently shown in \cite{Hoerr} that in two and three dimensions, for a stationary Boolean model with rotation regular grain distribution, the intensity can be expressed as a series of products of the densities $\overline{V}_j^{l,p}(Z)$ of the harmonic intrinsic volumes. For the proofs and the definition of the constants in the following two theorems we refer to \cite{Hoerr}.
\begin{theorem}
In two dimensions, the intensity $\gamma$ has the series representation
\begin{align*}
\gamma&=\rho\,\overline{V}_0(Z)+\rho^2\,
\sum\limits_{l,m=0}^\infty\sum\limits_{p=1}^{D(2,l)}\sum\limits_{q=1}^{D(2,m)} \; c_{l,m}^{p,q}\; \overline{V}_1^{\,l,p}(Z)\,\overline{V}_1^{\,m,q}(Z)
\end{align*}
with some constants $c_{l,m}^{p,q}\in\R$ and
\[
\rho:= \frac{1}{1-\overline{V}_2(Z)}.\]
\end{theorem}

\begin{theorem}
In three dimensions, the intensity $\gamma$ has the series representation
\begin{align*}
\gamma&= \rho\,\overline{V}_0(Z)+\rho^2 \sum\limits_{l,m=0}^\infty\sum\limits_{p=1}^{D(3,l)}\sum\limits_{q=1} ^{D(3,m)}\,d_{l,m}^{p,q}
\overline{V}_1^{\,l,p}(Z)\overline{V}_2^{\,m,q}(Z)\\
&\quad+\rho^3 \sum\limits_{l,m,o=0}^\infty\sum\limits_{p=1}^{D(3,l)}
\sum\limits_{q=1}^{D(3,m)} \sum\limits_{s=1}^{D(3,o)} e_{l,m,o}^{p,q,s}\, \overline{V}_2^{\,l,p}(Z) \overline{V}_2^{\,m,q}(Z)\overline{V}_2^{\,o,s}(Z)
\end{align*}
with some constants $d_{l,m}^{p,q},e_{l,m,o}^{p,q,s}\in \R$ and
\[
\rho:=\frac{1}{1-\overline{V}_3(Z)}.
\]
\end{theorem}

 These representations of the intensity can be seen as a generalization of the results by Miles and Davy for isotropic Boolean models from 1976 which were mentioned in the introduction and described in Section \ref{Sec:IsotropicCase}.
The article \cite{Hoerr2} will also contain applications of the series representation in Theorem 6 to specific examples of Boolean models. 

Harmonic intrinsic volumes are real-valued functionals but they are closely related to tensor-valued functionals. For the corresponding Minkowski tensors, see \cite{Kapfer, Mickel} and \cite{Klatt}.


\begin{thebibliography}{99.}

\bibitem{Fallert} Fallert, H.: Querma{\ss}dichten f\"ur Punktprozesse konvexer K\"orper und Boolesche Modelle. Math. Nachr. {\bf 181}, 37--48 (1998)

\bibitem{GHHRW} Goodey, P., Hinderer, W., Hug, D., Rataj, J., Weil, W.: A flag representation of projection functions. arXiv: 1502.06747 (2015)

\bibitem{GW03} Goodey, P., Weil, W.: Translative and kinematic integral
formulae for support functions II. Geom. Dedicata {\bf 99}, 103--125 (2003)


\bibitem{Had52} Hadwiger, H.: Translationsinvariante, additive und schwachstetige Polyederfunktionale.
Arch. Math. {\bf 3}, 387--394 (1952)

\bibitem{Hind} Hinderer, W.: Integral Representations of Projection Functions. PhD Thesis, University of Karlsruhe, Karlsruhe (2002)

\bibitem{HHW} Hinderer, W., Hug, D., Weil, W.: Extensions of translation invariant valuations on polytopes.  Mathematika {\bf 61}, 236--258 (2015)

\bibitem{Hoerr} H\"orrmann, J.: The method of densities for non-isotropic Boolean models. PhD Thesis, KIT Scientific Publishing, Karlsruhe 2015

\bibitem{Hoerr2} H\"orrmann, J.: Intensity estimation for non-isotropic Boolean models via harmonic intrinsic volumes. (in preparation)

\bibitem{HHKM} H\"orrmann, J., Hug, D., Klatt, M., Mecke, K.: Minkowski tensor density formulas for Boolean models. Adv. in Appl. Math. {\bf 55}, 48--85 (2014)

\bibitem{Hug} Hug, D.:   Measures, curvatures and currents in convex geometry.
Habilitation Thesis, University of Freiburg, Freiburg (1999)


\bibitem{HL} Hug, D., Last, G.: On support measures in Minkowski spaces and contact distributions in stochastic geometry. Ann. Probab. {\bf 28}, 796--850 (2000)


\bibitem{HS15} Hug, D., Schneider, R.: Tensor valuations and their local versions. In this volume (2015)

\bibitem{HTW} Hug, D., T\"urk, I., Weil, W.: Flag measures for convex bodies. In: Ludwig, M. et al. (eds.) Asymptotic Geometric Analysis, pp. 145--187. Fields Institute Communications, Vol. {\bf 68}, Springer (2013)

\bibitem{Kapfer} Kapfer, S.C.: Morphometry and Physics of Particulate and Porous Media.
PhD thesis, Friedrich-Alexander-Universit\"at Erlangen-N\"urnberg, Erlangen (2011).

\bibitem{KW} Kiderlen, M., Weil, W.: Measure-valued valuations and mixed
curvature measures of convex bodies.  Geom. Dedicata
{\bf 76}, 291--329 (1999)

\bibitem{Klatt} Klatt, M.A.: Morphometry of random spatial structures in physics. PhD Thesis, Friedrich-Alexander-Universit\"at Erlangen-N\"urnberg, Erlangen (2015).





\bibitem{McM77}  McMullen, P.: Valuations and Euler-type relations on certain classes of convex polytopes.
Proc. London Math. Soc. (3) {\bf 35}, 113--135 (1977)

\bibitem{McM80} McMullen, P.: Continuous translation-invariant valuations on the space of compact convex sets. Arch. Math. {\bf 34}, 377--384 (1980)



\bibitem{McM93} McMullen, P.: Valuations and dissections. In: Gruber, P.M., Wills, J.M. (eds) Handbook of Convex Geometry, vol. B, pp. 933--988, North-Holland, Amsterdam (1993)

\bibitem{Mickel} Mickel, W., Kapfer, S.C., Schr\"oder-Turk, G.E., Mecke, K.. Shortcomings
of the bond orientational order parameters for the analysis of disordered
particulate matter. J. Chem. Phys., {\bf 138}(4):044501, (2013).

\bibitem{Sch75} Schneider, R.: Kinematische Ber\"uhrma{\ss}e f\"ur konvexe K\"orper und Integralrelationen f\"ur Oberfl\"achenma{\ss}e. Math. Ann. {\bf 218}, 253--267 (1975)



\bibitem{Sch03} Schneider, R.: Mixed polytopes. Discrete Comput. Geom.  {\bf 29}, 575--593 (2003)

\bibitem{Sch14} Schneider, R.:{ Convex Bodies: The Brunn-Minkowski Theory}. 2nd Ed., Cambridge University Press, Cambridge (2014)

\bibitem{S2015} Schneider, R.: Valuations on convex bodies - the classical basic facts. In this volume (2015)

\bibitem{SW86} Schneider, R., Weil, W.: Translative and kinematic integral formulae for curvature measures. Math. Nachr. {\bf 129}, 67--80 (1986)

\bibitem{SW} Schneider, R., Weil, W.: {Stochastic and Integral Geometry}. Springer, Hei\-del\-berg-New York (2008)


\bibitem{W90} Weil, W.: Iterations of translative integral formulae and non-isotropic Poisson processes of particles. Math. Z. {\bf 205}, 531--549 (1990)

\bibitem{W95} Weil, W.: Translative and kinematic integral formulae for support
functions. Geom. Dedicata {\bf 57}, 91--103 (1995)

\bibitem{W2015a} Weil, W.: Integral geometry of translation invariant functionals, I: The polytopal case.  Adv. in Appl. Math. {\bf 66}, 46--79 (2015)

\bibitem{W2015b} Weil, W.: Integral geometry of translation invariant functionals, II: The case of general convex bodies. arXiv:1508.01142v1 (2015)

\bibitem{WW} Weil, W., Wieacker, J.A.: Densities for stationary random sets and point processes. Adv. Appl. Prob. {\bf 16}, 324--346 (1984)


\end{thebibliography}
\end{document}